\documentclass[smallextended]{svjour3}       

\usepackage{color}
\def\Add{}\def\endAdd{}

\smartqed  
\usepackage{mathptmx}      

\usepackage{amsmath,amsfonts,amssymb,amscd}

\usepackage{graphicx}
\graphicspath{{figs/}}

\newcommand\0{\phantom{0}}
\newcommand\dd{\partial}
\newcommand\threeVec[3]{\left(\begin{matrix}#1\\#2\\#3\end{matrix}\right)}

\newcommand\grad{\nabla}
\renewcommand\div{\nabla\cdot}
\newcommand\hatdiv{\hat\nabla\cdot}
\newcommand\Div{\text{{\rm div}}}
\newcommand\curl{\nabla\times}
\newcommand\Curl{\operatorname{curl}}
\newcommand\spn{\text{{\rm span}}}

\renewcommand\Re{\mathbb R}
\newcommand\Bu{\mathbb B}
\newcommand\Di{\mathbb D}
\newcommand\Ed{\mathbb E}
\newcommand\Po{\mathbb P}
\newcommand\Su{\mathbb S}

\newcommand\e{{\mathbf e}}
\newcommand\C{{\mathbf C}}
\newcommand\cc{{\mathbf c}}
\newcommand\hh{{\mathbf H}}
\newcommand\F{{\mathbf F}}
\newcommand\uu{{\mathbf u}}
\newcommand\vv{{\mathbf v}}
\newcommand\V{{\mathbf V}}
\newcommand\x{{\mathbf x}}
\newcommand\y{{\mathbf y}}
\newcommand\bpsi{{\boldsymbol\psi}}
\newcommand\hatbpsi{{\boldsymbol{\hat\psi}}}
\newcommand\bsigma{{\boldsymbol\sigma}}
\newcommand\hatbsigma{{\boldsymbol{\hat\sigma}}}
\newcommand\bphi{{\boldsymbol\phi}}
\newcommand\bvarphi{{\boldsymbol\varphi}}

\newcommand\cB{{\cal B}}
\newcommand\cF{{\cal F}}
\newcommand\cP{{\cal P}}
\newcommand\cPj{{\mathbf {Proj}}}
\newcommand\cT{{\cal T}}

\newcommand\red{\text{\rm red}}
\newcommand\ABF{\text{\rm ABF}}

\newcommand\Line[4]{\qbezier(#1,#2)(#1,#2)(#3,#4)}

\newcommand\myStrut{\vphantom{\int^H}}

\begin{document}

\title{Construction of $H(\Div)$-Conforming\\Mixed Finite Elements on
  Cuboidal Hexahedra\thanks{This work was supported by the
    U.S.~National Science Foundation under grant DMS-1418752.}}
\titlerunning{Mixed Finite Elements on Hexahedra}        

\author{
  Todd Arbogast
\and
  Zhen Tao
}

\institute{
  Todd Arbogast \at
  Department of Mathematics C1200,
  University of Texas,
  Austin, TX 78712-1202 and
  Institute for Computational Engineering and Sciences C0200,
  University of Texas,
  Austin, TX 78712-1229
 \email{arbogast@ices.utexas.edu}
\and
  Zhen Tao \at
  Institute for Computational Engineering and Sciences C0200,
  University of Texas,
  Austin, TX 78712-1229
  \email{taozhen.cn@gmail.com}
}

\date{Received: July 19, 2017 / Accepted: date} 

\maketitle

\begin{abstract}
  We generalize the two dimensional mixed finite elements of Arbogast and Correa [T. Arbogast and
  M. R. Correa, SIAM J. Numer. Anal., 54 (2016), pp. 3332--3356] defined on quadrilaterals to three
  dimensional cuboidal hexahedra.  The construction is similar in that polynomials are used directly
  on the element and supplemented with functions defined on a reference element and mapped to the
  hexahedron using the Piola transform.  The main contribution is providing a systematic procedure
  for defining supplemental functions that are divergence-free and have any prescribed polynomial
  normal flux.  \Add General procedures \endAdd are also presented for determining which
  supplemental \Add normal \endAdd fluxes are required to define the finite element space. Both full and reduced
  $H(\Div)$-approximation spaces may be defined, so the scalar variable, vector variable, and vector
  divergence are approximated optimally.  The spaces can be constructed to be of minimal local
  dimension, if desired. 
\end{abstract}


\section{Introduction}\label{sec:intro}

It is well-known that standard mixed finite elements defined on a square or cube and mapped to a general convex quadrilateral or cuboidal hexahedron perform poorly; in fact, they fail to approximate the divergence in an optimal way or require a very high number of local degrees of freedom.  Recently, Arbogast and Correa~\cite{Arbogast_Correa_2016} resolved the problem on quadrilaterals \Add (although, see the 2004 paper \cite{Duan_Liang_2004} for the lowest order case). \endAdd They defined two families of mixed finite elements that are of minimal local dimension and achieve optimal convergence properties.  In this paper, we generalize these elements to convex, cuboidal hexahedra, i.e., convex polyhedra with six flat quadrilateral faces.

It is convenient to discuss $H(\Div)$-conforming mixed finite elements in the context of the
simplest problem to which they apply.  Let $\Omega\subset\Re^d$, $d=2$ or $3$, be a polytopal
domain, let $W=L^2(\Omega)$ and $(\cdot,\cdot)_\omega$ denote the $L^2(\omega)$ or $(L^2(\omega))^d$
inner-product, and let $\V=H(\Div;\Omega)=\{\vv\in(L^2(\Omega))^2\,:\,\div\uu\in L^2(\Omega)\}$.
Consider the second order elliptic boundary value problem in mixed variational form: Find
$(\uu,p)\in\V\times W$ such that
\begin{alignat}2
\label{eq:var_darcy}
(a^{-1}\uu,\vv)_\Omega - (p,\div\vv)_\Omega &= 0&&\quad\forall\vv\in\V,\\
\label{eq:var_conservation}
\quad(\div\uu,w)_\Omega &= (f,w)_\Omega&&\quad\forall w\in W,
\end{alignat}
where $f\in L^2(\Omega)$ and the tensor $a$ is uniformly positive definite and bounded.  A mixed
finite element method is given by restricting $\V\times W$ to inf-sup compatible finite element
subspaces $\V_r\times W_r\subset\V\times W$ defined (in our case) over a mesh of convex, cuboidal
hexahedra, where $r\ge0$ is the index of the subspaces.

Full $H(\Div)$-approximation spaces of index $r\ge0$ approximate $\uu$, $p$, and $\div\uu$ to order
$h^{r+1}$, where $h$ is the maximal diameter of the computational mesh elements.  Such spaces
include the classic spaces of Raviart-Thomas (RT)~\cite{Raviart_Thomas_1977,Thomas_1977} in 2-D and
3-D, as well as, in 2-D only, the spaces of Arnold-Boffi-Falk (ABF)~\cite{ABF_2005} and
Arbogast-Correa (AC)~\cite{Arbogast_Correa_2016}.  The ABF spaces have been generalized recently to
3-D by Bergot and Durufle~\cite{Bergot_Durufle_2013}.  Reduced $H(\Div)$-approximation spaces of
index $r\ge1$ approximate $\uu$ to order $h^{r+1}$ and $p$ and $\div\uu$ to order $h^{r}$.  In this
category are the classic spaces due to Brezzi-Douglas-Marini (BDM)~\cite{BDM_1985} in 2-D and their
3-D counterpart from Brezzi-Douglas-Dur{\`{a}}n-Fortin (BDDF)~\cite{BDDF_1987,Arnold_Awanou_2014},
as well as the reduced Arbogast-Correa (AC$^\red$) spaces~\cite{Arbogast_Correa_2016} in 2-D.
Recent progress on defining 3-D mixed finite elements has been made by many authors, including, but
certainly not exhaustively,
\cite{Falk_Gatto_Monk_2011,Bergot_Durufle_2013,Arnold_Awanou_2014,Arnold_Boffi_Bonizzoni_2015,Cockburn_Fu_2017_mDecompIII}.

All spaces save AC, AC$^\red$, and the spaces of Cockburn and Fu~\cite{Cockburn_Fu_2017_mDecompIII}
are defined on a reference square or cube $\hat E=[0,1]^d$ and mapped to the element $E$ using the
Piola transform.  The RT and BDM (and BDDF) spaces lose accuracy.  The ABF spaces maintain accuracy,
but at the expense of adding many extra degrees of freedom to the local finite element space.
Cockburn and Fu construct finite elements on hexahedra using a sub mesh of tetrahedra.

The two families of AC spaces, $\V_r$ and $\V_r^{\red}$, are constructed using a different
strategy.  They use polynomials defined directly on the element and supplemented by two (one if
$r=0$) basis functions defined on a reference square and mapped via Piola.  Let $\Po_r$ denote the
space of polynomials of degree up to $r$, and let $\tilde\Po_r$ denote the space of homogeneous
polynomials of exact degree $r$.  On a convex quadrilateral element $E$, for which $d=2$ and
$\x=(x_1,x_2)$, the full $H(\Div)$-approximation spaces of index $r\ge0$ are
\begin{equation}
\label{eq:fullSpaces}
\V_r(E) = (\Po_r)^d\oplus\x\tilde\Po_r\oplus\Su_r(E)
\quad\text{and}\quad
W_r(E) = \Po_r,
\end{equation}
and the reduced $H(\Div)$-approximation spaces of index $r\ge1$ are
\begin{equation}
\label{eq:reducedSpaces}
\V_r^\red(E) = (\Po_r)^d\oplus\Su_r(E)
\quad\text{and}\quad
W_r(E) = \Po_{r-1}.
\end{equation}
One can define the reference supplemental space on $\hat E=[0,1]^2$ in 2-D as
\begin{equation}
\label{eq:ac2dsuppHat}
\hat\Su_r = \begin{cases}
\text{span}\big\{\widehat\Curl\big((\hat x_1-1/2)(\hat x_2-1/2)\big)\big\},&r=0,\\
\text{span}\big\{
\widehat\Curl\big((\hat x_1-1/2)^{r-1}\hat x_1(1-\hat x_1)(\hat x_2-1/2)\big),\\
\qquad\ \ \widehat\Curl\big((\hat x_1-1/2)(\hat x_2-1/2)^{r-1}\hat x_2(1-\hat x_2)\big)
\big\},&r\ge1,
\end{cases}
\end{equation}
and then
\begin{equation}
\label{eq:supplementsMapped}
\Su_r(E) = \cP_E\,\hat\Su_r,
\end{equation}
where $\cP_E$ is the Piola transform from $\hat E=[0,1]^d$ to $E$.

Our generalization of the two families of AC spaces \Add to the case of a convex, cuboidal hexahedron $E$ gives  full and reduced $H(\Div)$-approximating mixed finite elements $\V(E)\times W(E)$ and $\V^\red(E)\times W(E)$, respectively.  These are defined to include spaces of polynomials and special supplemental functions.  In fact, the spaces are defined formally by the same equations \endAdd \eqref{eq:fullSpaces}--\eqref{eq:reducedSpaces}, \eqref{eq:supplementsMapped}, \Add except that \endAdd now $d=3$, $\x=(x_1,x_2,x_3)$, \Add and the supplemental space $\Su_r(E)$ or $\hat\Su_r$ (replacing \eqref{eq:ac2dsuppHat}) must be defined carefully. \endAdd The number of supplemental functions is 2 for $r=0$ and otherwise at most $3(r+1)$.  The divergences of these vectors lie in $\Po_{r}$ for the full space and in $\Po_{r-1}$ for the
reduced space, and the normal flux on each edge or face~$f$ of~$E$ is in $\Po_r(f)$ (i.e., $\Po_r$ in dimension~$d-1$). In fact, the degrees of freedom (DOFs) of a vector $\vv\in\V_r$ or $\V_r^\red$ include the divergence and edge or face normal fluxes:
\begin{alignat}2
\label{eq:div_DOFs}
&(\div\vv,w)_E&&\quad\forall w\in\Po_{r}^*\ (\text{for }\V_r) \text{ or }\Po_{r-1}^*\ (\text{for }\V_r^\red),\\
\label{eq:flux_DOFs}
&(\vv\cdot\nu,\mu)_f&&\quad\forall\text{\,edges ($d=2$) or faces ($d=3$) }f\text{ of }E
           \text{ and }\forall\mu\in\Po_{r}(f),
\end{alignat}
where $\nu$ is the outer unit normal vector to $E$ and $\Po_{r}^*$ are the polynomials of degree $r$
with no constant term.  The purpose of the supplements is to make these DOFs independent, so that
the elements can be joined in $H(\Div)$ to form $\V_r$ or $\V_r^\red$ while also maintaining
consistency to approximate the divergence.  The set of DOFs is completed by adding conditions on the
interior, divergence-free, bubble functions (for $H(\Div)$-conforming elements, an {\em interior
  bubble function\/} is a vector function with vanishing normal component on $\dd E$).

After setting some additional notation in Section~\ref{sec:notation}, we describe how to construct
arbitrary, divergence-free supplemental functions in 3-D with a prescribed normal flux in
Sections~\ref{sec:presupplements} and~\ref{sec:supplements}.  In Section~\ref{sec:atSpaces}, we
describe a way to choose the specific supplemental function \Add space $\hat\Su_r$ \endAdd needed to
define $\Su_r(E)$ \Add by \eqref{eq:supplementsMapped}. \endAdd The most useful cases $r=0$ and
$r=1$ are given in detail \Add (although some proofs are relegated to the appendices). \endAdd For
$r\ge1$, we need to determine the normal fluxes needed to ensure that the DOFs \eqref{eq:flux_DOFs}
are independent. We note the recent work of Cockburn and Fu~\cite{Cockburn_Fu_2017_mDecompIII} in
this regard, but we provide a method for resolving this issue based on linear algebra. \Add We
present \endAdd some numerical results in Section~\ref{sec:numerics}.  \Add We close by summarizing
our results in the last section. \endAdd


\section{Further notation}\label{sec:notation}

In this section, we fix the notation and geometry used throughout the paper. As noted above, let
$\Po_r$ denote the space of polynomials of degree~$r$.  Generally, $\Po_r=\Po_r(\Re^3)$ is defined
over a three-dimensional domain. Sometimes we need to restrict polynomials to faces, so let
$\Po_r(f)$ be the polynomials defined over the domain~$f$.  Let $\tilde\Po_r$ denote the space of
homogeneous polynomials of degree~$r$.  We also let $\Po_{r,s,t}$ denote the tensor product
polynomial spaces of degree~$r$ in $x_1$, $s$ in $x_2$, and $t$ in $x_3$.

\subsection{A convex, cuboidal hexahedron and the Piola map}\label{sec:hexahedron}

Fix the reference element $\hat E=[0,1]^3$ and take any convex, cuboidal hexahedron $E$ oriented as
in Figure~\ref{fig:geometry}.  The reference element $\hat E$ has faces ordered as follows.  Face~0
is where $\hat x_1=0$ and it is denoted $\hat f_0=E\cap\{\hat x_1=0\}$, face~1 is where $\hat x_1=1$
and it is denoted $\hat f_1$, and so forth to face~5 is where $\hat x_3=1$ and it is denoted
$\hat f_5$.  The vertices $\hat\x_{ijk}$ are indexed by the faces of intersection, i.e.,
$\hat\x_{ijk}=\hat f_i\cap\hat f_j\cap\hat f_k$. The bijective and trilinear map
$\F_{\!E}:\hat E\to E$ is defined by
\begin{align}\nonumber
\F_{\!E}(\hat\x) 
&= \x_{024}(1-\hat x_1)(1-\hat x_2)(1-\hat x_3) + \x_{124}\,\hat x_1(1-\hat x_2)(1-\hat x_3)\\\nonumber
&\quad + \x_{034}(1-\hat x_1)\hat x_2(1-\hat x_3) + \x_{134}\,\hat x_1\hat x_2(1-\hat x_3)\\\nonumber
&\quad + \x_{025}(1-\hat x_1)(1-\hat x_2)\hat x_3 + \x_{125}\,\hat x_1(1-\hat x_2)\hat x_3\\\nonumber
&\quad + \x_{035}(1-\hat x_1)\hat x_2\hat x_3 + \x_{135}\,\hat x_1\hat x_2\hat x_3\\
&\in \Po_{1,1,1}.\label{eq:F_E}
\end{align}
This map fixes the notation on $E$ (faces $f_i=\F_{\!E}(\hat f_i)$ and vertices
$\x_{ijk}=\F_{\!E}(\hat\x_{ijk})$).
The center of face~$i$ is denoted $\x_i$. The outer unit normal to face~$i$ is 
\Add$\nu_i=(\nu_{i,1}, \nu_{i,2}, \nu_{i,3})$.  \endAdd
For example,
\begin{equation}
\nu_1 = \frac{(\x_{134}-\x_{124})\times(\x_{125}-\x_{124})}{\|(\x_{134}-\x_{124})\times(\x_{125}-\x_{124})\|}.
\end{equation}

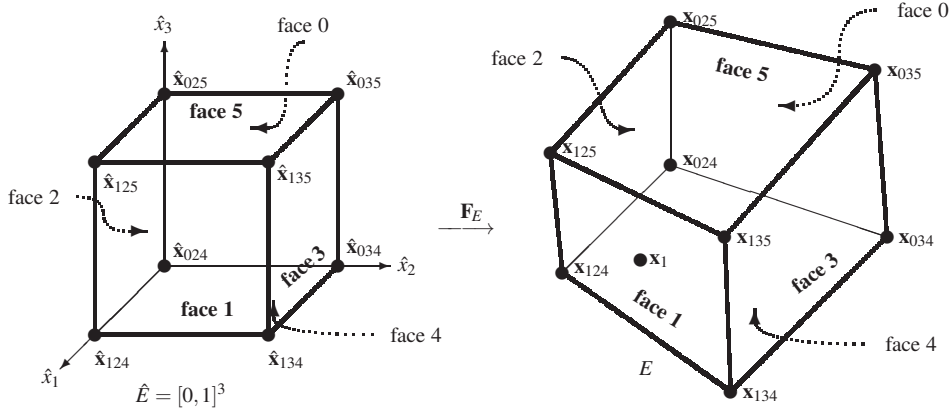
\begin{figure}\centerline{{\setlength\unitlength{6.5pt}
\begin{picture}(26,24)(-5.5,-4)
%
\put(4,4){\circle*{0.7}}\put(4.5,4.5){$\hat\x_{024}$}
\put(0,0){\circle*{0.7}}\put(0,-1.75){$\hat\x_{124}$}
\put(14,4){\circle*{0.7}}\put(14.5,4.75){$\hat\x_{034}$}
\put(10,0){\circle*{0.7}}\put(10,-1.75){$\hat\x_{134}$}
\put(4,14){\circle*{0.7}}\put(4.5,14.5){$\hat\x_{025}$}
\put(0,10){\circle*{0.7}}\put(0.5,8.5){$\hat\x_{125}$}
\put(14,14){\circle*{0.7}}\put(14.5,14.5){$\hat\x_{035}$}
\put(10,10){\circle*{0.7}}\put(10.5,9){$\hat\x_{135}$}
\thinlines
\put(4,4){\vector(-1,-1){6}}\put(-2.5,-2.5){\makebox(0,0){$\hat x_1$}}
\put(4,4){\vector(1,0){13}}\put(4,18){\makebox(0,0){$\hat x_3$}}
\put(4,4){\vector(0,1){13}}\put(18,4){\makebox(0,0){$\hat x_2$}}
\put(4,4){\line(0,1){10}}
\put(4,4){\line(1,0){10}}
\put(4,4){\line(-1,-1){4}}
\linethickness{1.25pt}
\Line{0}{10}{4}{14}
\put(0,10){\line(1,0){10}}
\put(0,10){\line(0,-1){10}}
\put(10,0){\line(-1,0){10}}
\put(10,0){\line(0,1){10}}
\Line{10}{0}{14}{4}
\put(14,14){\line(-1,0){10}}
\put(14,14){\line(0,-1){10}}
\Line{14}{14}{10}{10}
\thinlines
\put(6.5,1.5){\makebox(0,0){\bf face~1}}
\put(12,4){\rotatebox{40}{\makebox(0,0){\sl\bf face\,3}}}
\put(7,13){\makebox(0,0){\sl\bf face~5}}
\thicklines
\qbezier[10](12,17)(11,17)(11,14)\qbezier[10](11,14)(11,12)(10,12)\put(10,12){\vector(-1,0){1}}\put(12,18){\makebox(0,0){face~0}}
\qbezier[20](16,0)(10.3,0)(10.3,1)\put(10.3,1){\vector(0,1){1}}\put(18.5,0){\makebox(0,0){face~4}}
\qbezier[8](-1,8)(0.5,8)(0.5,7)\qbezier[8](0.5,7)(0.5,6)(2,6)\put(2,6){\vector(1,0){1}}\put(-3.5,8){\makebox(0,0){face~2}}
\put(5,-3.5){\makebox(0,0){$\hat E=[0,1]^3$}}
\end{picture}}%
\!\!\!\!\raisebox{64pt}{$-\!\!\!-\!\!\!\longrightarrow$\llap{\raisebox{6pt}{$\F_{\!E}$\ \ }}}\!\!\!\!%
{\setlength\unitlength{4.5pt}\begin{picture}(39.5,34)(-5.8,-1)
%
\put(10,19){\circle*{1}}\put(11,19){$\x_{024}$}
\put(1,10){\circle*{1}}\put(2,10){$\x_{124}$}
\put(28,13){\circle*{1}}\put(29,12.5){$\x_{034}$}
\put(15,0){\circle*{1}}\put(16,-0.5){$\x_{134}$}
\put(10,31){\circle*{1}}\put(11,31){$\x_{025}$}
\put(0,20){\circle*{1}}\put(1,20){$\x_{125}$}
\put(27,27){\circle*{1}}\put(28,26.5){$\x_{035}$}
\put(14.5,13){\circle*{1}}\put(15.5,12.5){$\x_{135}$}
\put(7.5,11){\circle*{1}}\put(8.25,10.75){$\x_{1}$}
\thinlines
\put(10,19){\line(0,1){12}}
\put(10,19){\line(-1,-1){9}}
\put(10,19){\line(3,-1){18}}
\linethickness{1.25pt}
\Line{0}{20}{1}{10}
\Line{0}{20}{10}{31}
\Line{0}{20}{14.5}{13}
\Line{15}{0}{1}{10}
\Line{15}{0}{28}{13}
\Line{15}{0}{14.5}{13}
\Line{27}{27}{28}{13}
\Line{27}{27}{10}{31}
\Line{27}{27}{14.5}{13}
\thinlines
\put(9,7){\rotatebox{-32}{\makebox(0,0){\bf face~1}}}
\put(22,10){\rotatebox{38}{\makebox(0,0){\sl\bf face~3}}}
\put(16,27){\rotatebox{-12}{\makebox(0,0){\sl\bf face~5}}}
\thicklines
\qbezier[14](27,32)(24,32)(24,28)\qbezier[14](24,28)(24,24)(20,24)\put(20,24){\vector(-1,0){1}}\put(31,32){\makebox(0,0){face~0}}
\qbezier[24](26,4)(17,4)(17,6)\put(17,6){\vector(0,1){1}}\put(30,4){\makebox(0,0){face~4}}
\qbezier[14](1,28)(3,28)(4,25)\qbezier[14](4,25)(4.5,22)(6,22)\put(6,22){\vector(1,0){1}}\put(-2.75,28){\makebox(0,0){face~2}}
\put(8,2){\makebox(0,0){$E$}}
\end{picture}}}
\caption{The geometry of the cuboidal hexahedron.  On the left is the reference $\hat E=[0,1]^3$, which is
  trilinearly mapped to the hexahedron~$E$. The faces are labeled from 0 to 5, and faces 1, 3, and 5
  are in front.  The corner points are labeled by their intersections with the faces (e.g.,
  $\x_{135}$ intersects faces 1, 3, and 5).
  The centers of the faces are labeled by the face (we show only $\x_1$ on face~1).}\label{fig:geometry}
\end{figure}

\subsubsection{Piola transform and Jacobians}
Let $D\F_{\!E}(\hat\x)$ denote the Jacobian matrix of $F_E$ and $J_E(\hat\x)=\det(D\F_{\!E}(\hat\x))$.  The
contravariant Piola transform $\cP_E$ maps a vector $\hat\vv:\hat E\to\Re^2$ to a vector
$\vv:E\to\Re^2$ by the formula
\begin{equation}
\label{eq:Piola}
\vv(\x) = \cP_E(\hat\vv)(\x) = \frac1{J_E}D\F_{\!E}\hat\vv(\hat\x),\quad\text{where }\x=\F_{\!E}(\hat\x).
\end{equation}
For a scalar function $w$, we define the map $\hat w$ by $\hat w(\hat\x)=w(\x)$, where again
$\x=\F_{\!E}(\hat\x)$.

The Piola transform preserves the divergence and normal components of $\hat\vv$ in the sense that
\begin{align}
\label{eq:Piola_div}
\div\vv  &= \frac1{J_E}\hatdiv\hat\vv,\\
\label{eq:Piola_normal}
\vv\cdot\nu &= \frac1{K_i}\hat\vv\cdot\hat\nu\quad\text{for each face $f_i$ of $\dd E$},
\end{align}
where $K_i$ is the face Jacobian.  The face Jacobian for face~$i$ is
\begin{align}
\label{eq:faceJacDef}
K_i = \bigg\|\Big(\frac{\dd\F_{\!E}}{\dd\hat x_\ell} \times \frac{\dd\F_{\!E}}{\dd\hat x_m}\Big)\Big|_{f_i}\bigg\|
= \bigg|\Big(\frac{\dd\F_{\!E}}{\dd\hat x_\ell} \times \frac{\dd\F_{\!E}}{\dd\hat x_m}\Big)\Big|_{f_i}\!\!\cdot \nu_{i}\bigg|,
\end{align}
where $i$, $\ell$, and $m$ are distinct integers from $\{1,2,3\}$ and, say, $\ell<m$.  The face Jacobian
describes the bilinear distortion of the face, and it depends only on the face vertices (so two
elements intersecting at face~$f$ will have the same face Jacobian). If we re-index the face so that
\begin{align}\nonumber
\F_{\!E}(\hat x_\ell,\hat x_m) \big|_{f_i} &= \y_{0}(1-\hat x_\ell)(1-\hat x_m) + \y_{1}\,\hat x_\ell(1-\hat x_m)\\
&\quad + \y_{2}(1-\hat x_\ell)\hat x_m + \y_{3}\,\hat x_\ell\,\hat x_m,
\end{align}
then it is not hard to show, when $f_i$ is flat, that
\begin{alignat}2\nonumber
K_i(\hat x_\ell,\hat x_m) &= &&\,\|(\y_2-\y_0)\times(\y_1-\y_0)\|(1-\hat x_\ell)(1-\hat x_m)\\\nonumber
&&+ &\,\|(\y_3-\y_1)\times(\y_0-\y_1)\|\,\hat x_\ell(1-\hat x_m)\\\nonumber
&&+ &\,\|(\y_3-\y_2)\times(\y_0-\y_2)\|(1-\hat x_\ell)\hat x_m\\\nonumber
&&+ &\,\|(\y_2-\y_3)\times(\y_1-\y_3)\|\,\hat x_\ell\,\hat x_m\\
&\in\rlap{$\Po_{1,1}$.}\label{eq:faceJac}
\end{alignat}

\subsubsection{Local variables}
It is clear that for the reference cube $\hat E$, the local variables can be taken as $\hat x_2$ and
$\hat x_3$ on faces 0 and 1, $\hat x_1$ and $\hat x_3$ on faces 2 and 3, and $\hat x_1$ and
$\hat x_2$ on faces 4 and 5.  Similar indexing does not necessarily hold on $E$. In fact, faces
indexed as being opposite to each other may be far from parallel (they could even be perpendicular
to each other).

It is necessary to select local variables on each face of $E$, two from among the set of variables
$\{x_1,x_2,x_3\}$.  For face~$\ell$, we denote these variables by $(x_{i_\ell},x_{j_\ell})$, where
we tacitly assume that $i_\ell<j_\ell$.  In practice, one can find the maximal absolute component of
$\nu_\ell$, say $|\nu_{\ell,m}|$, and omit $x_m$ from the set $\{x_1,x_2,x_3\}$, leaving the local
coordinates $\{x_{i_\ell},x_{j_\ell}\}$.


\section{Construction of Pre-supplemental Functions on the Reference Cube}\label{sec:presupplements}

In this section, we construct a vector function on the reference cube $\hat E=[0,1]^3$ with a
vanishing divergence and prescribed monomial normal flux (up to a constant).  These functions will
be used later to construct the space of supplements $\Su_r(E)$ for the new mixed finite elements.
We call our special vector functions {\em pre-supplements.}  For simplicity, we consider only face~1
(where $\hat x_1=1$).  The other faces are handled analogously.

\Add The vector functions in the local BDDF spaces of index $r$~\cite{BDDF_1987,Arnold_Awanou_2014}
have the property that their normal fluxes are polynomials of degree~$r$.  Moreover, both the normal
fluxes and the divergence are degrees of freedom. Analogous to BDDF, we can define vector functions
with the properties we desire.
\endAdd
Let us fix the monomial as $\hat x_2^\ell\hat x_3^m$ for some integers $\ell\ge0$ and $m\ge0$.  We define the pre-supplement to be, when $\ell+m\ge1$,
\begin{equation}\label{eq:presupp1}
\hatbpsi_{\ell,m}^{1} = \left(\begin{matrix}
\hat x_1\hat x_2^\ell\hat x_3^m - \dfrac{\hat x_1}{(\ell+1)(m+1)}\\[8pt]
\dfrac1{2(\ell+1)}\hat x_2(1-\hat x_2^\ell) \Big(\hat x_3^m + \dfrac1{m+1}\Big)\\[10pt]
\dfrac1{2(m+1)}\hat x_3(1-\hat x_3^m)\Big(\hat x_2^\ell + \dfrac1{\ell+1}\Big)
\end{matrix}\right)
\in \Po_{\ell+m+1}^3(\hat E).
\end{equation}
\Add 
It can be readily verified that indeed this function lies in the more symmetric BDDF space as defined by Arnold and Awanou~\cite{Arnold_Awanou_2014}, although this fact is not important in itself. What is important is that we have our desired properties \endAdd
\begin{align}\label{eq:presupp1_div_flux}
\hatdiv\hatbpsi_{\ell,m}^{1} = 0
\quad\text{and}\quad
\hatbpsi_{\ell,m}^{1}\cdot\hat\nu &= \begin{cases}
\hat x_2^\ell\hat x_3^m - \dfrac1{(\ell+1)(m+1)}&\text{on }\hat f_1,\ \ell+m\ge1,\\
0&\text{on }\hat f_i,\ i=0,2,\ldots,5,
\end{cases}
\end{align}
where we recall that the face~$f_1$ is where $\hat x_1=1$.  The case $\ell=m=0$ reduces to the zero
vector because of the divergence theorem.  We therefore accept a constant divergence and simply take
\begin{equation}\label{eq:presupp1-00}
\hatbpsi_{0,0}^{1} = \left(\begin{matrix}\hat x_1\\0\\0\end{matrix}\right),
\end{equation}
for which
\begin{align}\label{eq:presupp1_00_div_flux}
\hatdiv\hatbpsi_{0,0}^{1} = 1
\quad\text{and}\quad
\hatbpsi_{0,0}^{1}\cdot\hat\nu &= \begin{cases}
1&\text{on }\hat f_1,\\
0&\text{on }\hat f_i,\ i=0,2,\ldots,5.
\end{cases}
\end{align}
We can construct similar pre-supplements for each face; label these as $\hatbpsi_{\ell,m}^i$ for face $i=0,1,\ldots,5$.

We remark that our pre-supplemental functions are not unique when there are divergence-free bubble
functions.  For example, to $\hatbpsi_{\ell,m}^1$, one could add any function of the form
\begin{equation}\label{eq:presupp_divFreeBubble}
\left(\begin{matrix}
0\\
\dfrac{\dd}{\dd\hat x_3}\big[\hat x_2(1-\hat x_2)\hat x_3(1-\hat x_3)\hat p\big]\\[10pt]
-\dfrac{\dd}{\dd\hat x_2}\big[\hat x_2(1-\hat x_2)\hat x_3(1-\hat x_3)\hat p\big]
\end{matrix}\right),
\end{equation}
where $\hat p$ is any polynomial in $\hat x_2$ and $\hat x_3$, and we would maintain
\eqref{eq:presupp1_div_flux}.


\section{Construction of the Supplemental Functions on Hexahedra}\label{sec:supplements}

In this section, we construct a supplemental vector function $\bsigma$ with zero divergence on the
convex, cuboidal hexahedron~$E$.  It has a prescribed polynomial normal flux (up to a constant) on a
single face and vanishing normal flux on the other 5 faces. We continue to fix the nonzero flux on
face~1 for ease of exposition; the other faces are handled similarly.  In terms of the local face
variables $(x_{i_1},x_{j_1})$, suppose that the prescribed flux is $x_{i_1}^\ell x_{j_1}^m$.  That
is, we want to define $\bsigma_{\ell,m}^1$ when $\ell+m\ge1$ so that, for some
constant~$c_{\ell,m}^1$,
\begin{align}
\label{eq:desired_bsigma}
\div\bsigma_{\ell,m}^1 = 0
\quad\text{and}\quad
\bsigma_{\ell,m}^1\cdot\nu &= \begin{cases}
x_{i_1}^\ell x_{j_1}^m - c_{\ell,m}^1&\text{on }f_1,\ \ell+m\ge1,\\
0&\text{on }f_i,\ i=0,2,\ldots,5.
\end{cases}
\end{align}

The construction is given by first defining an appropriate vector function $\hat\bsigma_{\ell,m}^1$
on the reference cube $\hat E$ and then mapping it to $E$ using the Piola transform
\eqref{eq:Piola}, so that $\bsigma_{\ell,m}^1=\cP_E\hat\bsigma_{\ell,m}^1$.  The key is to recognize
that the normal components of $\hat\bsigma_{\ell,m}^1$ transform by \eqref{eq:Piola_normal}, and
therefore we need to include the factor $K_1$ within the first row of $\hat\bsigma_{\ell,m}^1$.  Our
construction is vaguely reminiscent of the one given in 2-D by Shen~\cite{Shen_1994} \Add (for which the
resulting method was later proved in \cite{Kwak_Pyo_2011}). \endAdd

To proceed, we must realize two simple facts.  First, the face Jacobian $K_1$ is bilinear in the
reference variables, i.e., \eqref{eq:faceJac} holds.  Second, the polynomial flux
$x_{i_1}^\ell x_{j_1}^m$ is evaluated in terms of the reference variables by the map $F_E:\hat E\to E$
\eqref{eq:F_E}, i.e.,
\begin{equation}
\label{eq:reference_monomial}
x_{i_1} = F_{i_1}(1,\hat x_2,\hat x_3)\quad\text{and}\quad x_{j_1} = F_{j_1}(1,\hat x_2,\hat x_3),
\end{equation}
which are both bilinear.  Therefore the product $x_{i_1}^\ell x_{j_1}^m$, multiplied by $K_1$ and
written in terms of the reference variables, is in the space $\Po_{n+1,n+1}$, where $n=\ell+m$.  Let
the pre-image of $x_{i_1}^\ell x_{j_1}^m$ (scaled by $K_1$) be denoted
\begin{align}\nonumber
K_1 x_{i_1}^\ell x_{j_1}^m &= K_1(\hat x_2,\hat x_3)\,F_{i_1}(1,\hat x_2,\hat x_3)^\ell\,F_{j_1}(1,\hat x_2,\hat x_3)^m\\
&= \underset{i+j\ge1}{\sum_{i=0}^{n+1}\sum_{j=0}^{n+1}}
          \alpha_{ij}^{\ell,m}\Big(\hat x_2^i\hat x_3^j - \frac1{(i+1)(j+1)}\Big) + \alpha_{0,0}^{\ell,m}.
\label{eq:referenceFlux}
\end{align}
That is, in practice, we compute the coefficients $\alpha_{ij}^{\ell,m}$ based on the geometry of
the hexahedron.

When $n=\ell=m=0$, let
\begin{equation}
\label{eq:bsigma1_0_0}
\hatbsigma_{0,0}^1 = \sum_{i=0}^{1}\sum_{j=0}^{1}\alpha_{ij}^{0,0}\hatbpsi_{i,j}^1.
\end{equation}
Recalling \eqref{eq:presupp1_div_flux} and \eqref{eq:presupp1_00_div_flux}, this function has
divergence $\alpha_{0,0}^{0,0}$ and flux $K_1$ on face~1.  By the divergence theorem, clearly
$\alpha_{0,0}^{0,0}=|f_1|$, the area of face~1, so
\begin{align}
\label{eq:bsigma0_div_flux}
\hat\nabla\cdot\hatbsigma_{0,0}^1 = |f_1|
\quad\text{and}\quad
\hatbsigma_{0,0}^1\cdot\hat\nu &= \begin{cases}
K_1&\text{on }\hat f_1,\\
0&\text{on }\hat f_i,\ i=0,2,\ldots,5.
\end{cases}
\end{align}

When $n=\ell+m\ge1$, we define
\begin{equation}
\label{eq:bsigma1_ell_m}
\hatbsigma_{\ell,m}^1
 = \sum_{i=0}^{n+1}\sum_{j=0}^{n+1}\alpha_{ij}^{\ell,m}\hatbpsi_{i,j}^1 - \frac{\alpha_{0,0}^{\ell,m}}{|f_1|}\hatbsigma_{0,0}^1,
\end{equation}
which has vanishing divergence and matches the flux \eqref{eq:referenceFlux}, up to a constant
multiple of~$K_1$.  Owing to \eqref{eq:Piola_div}--\eqref{eq:Piola_normal},
$\bsigma_{\ell,m}^1=\cP_E\hat\bsigma_{\ell,m}^1$ has the desired properties
\eqref{eq:desired_bsigma}. We can construct a similar vector function for each face; label these as
$\bsigma_{\ell,m}^i$ for face $i=0,1,\ldots,5$.

In the case of constant normal face fluxes (i.e., $n=0$), we cannot remove the divergence unless
we allow nonzero flux on at least two faces. We \Add therefore define and later use the lowest order \endAdd  divergence-free supplements \Add given by \endAdd
\begin{align}
\label{eq:bsigma0}
\bsigma_{0,0}^{i,j} &= \cP_E\Big(\frac{\hatbsigma_{0,0}^i}{|f_i|} - \frac{\hatbsigma_{0,0}^j}{|f_j|}\Big).
\end{align}
\Add
Using (\ref{eq:Piola_div}), (\ref{eq:Piola_normal}) and
(\ref{eq:bsigma0_div_flux}), it can be easily verified that $\bsigma_{0,0}^{i,j}$ is divergence-free and provides constant
normal fluxes on faces $i$ and $j$.
\endAdd


\section{Generalized AC Spaces on Convex, Cuboidal Hexahedra}\label{sec:atSpaces}

We now present our generalization of the two families of AC spaces~\cite{Arbogast_Correa_2016}.  The
full and reduced spaces are given by \eqref{eq:fullSpaces} and \eqref{eq:reducedSpaces},
respectively, once we have defined the supplemental space $\Su_r$ for $r\ge0$, so that the DOFs
\eqref{eq:div_DOFs}--\eqref{eq:flux_DOFs} are independent.

The supplemental space is constructed using the functions defined in
Sections~\ref{sec:presupplements}--\ref{sec:supplements}, once we know what fluxes are required to
independently span the space of \Add normal \endAdd fluxes \eqref{eq:flux_DOFs}.  To this end, it is convenient to
define the \Add full \endAdd flux operator $\cF$ \Add as well as the operators $\cF_{024}$ and $\cF_{135}$ on the even and odd faces, respectively, \endAdd to be
\begin{alignat}2
\label{eq:fluxOp}
\cF(\uu) &= \big[\uu\cdot\nu_0|_{f_0},\ \ldots\ \uu\cdot\nu_5|_{f_5}\big]
&&\subset\prod_{i=0}^5\Po_r(f_i) = (\Po_r(\Re^2))^{1\times6},\\
\nonumber
\Add \cF_{024}(\uu) \endAdd
&= \big[\uu\cdot\nu_0|_{f_0},\uu\cdot\nu_2|_{f_2},\uu\cdot\nu_4|_{f_4}\big]
&&\subset\prod_{i=0}^2\Po_r(f_{2i}) = (\Po_r(\Re^2))^{1\times3},\\
\nonumber
\Add \cF_{135}(\uu) \endAdd
&= \big[\uu\cdot\nu_1|_{f_1},\uu\cdot\nu_3|_{f_3},\uu\cdot\nu_5|_{f_5}\big]
&&\subset\prod_{i=0}^2\Po_r(f_{2i+1}) = (\Po_r(\Re^2))^{1\times3},
\end{alignat}
\Add Note that $\cF$ is a permutation of the block matrix $\big[\cF_{024}~~\cF_{135}\big]$. \endAdd For a sequence of $n$ functions, we also define the ``flux matrix'' as
\begin{equation}\label{eq:fluxMatrix}
\cF(\uu_1,\ldots,\uu_n) = \left[\begin{matrix}
\cF(\uu_1)\\
\vdots\\
\cF(\uu_n)
\end{matrix}\right] = \left[\begin{matrix}
\uu_1\cdot\nu_0|_{f_0}&\ldots&\uu_1\cdot\nu_5|_{f_5}\\
\vdots&\ddots&\vdots\\
\uu_n\cdot\nu_0|_{f_0}&\ldots&\uu_n\cdot\nu_5|_{f_5}
\end{matrix}\right] \in(\Po_r(\Re^2))^{n\times6},
\end{equation}
\Add and we define $\cF_{024}(\uu_1,\ldots,\uu_n)$ and $\cF_{135}(\uu_1,\ldots,\uu_n)$ in $(\Po_r(\Re^2))^{n\times3}$ analogously. \endAdd

\subsection{The case $r=0$}\label{sec:at0}

On the convex, cuboidal hexahedron $E$, the new space is
\begin{equation}\label{eq:at0}
\V_0(E) = \Po_0^3\oplus\x\Po_0\oplus\Su_0,
\end{equation}
which has only normal flux DOFs.  We will give two definitions of $\Su_0$, but first, note that
$\Po_0^3\oplus\x\Po_0$ has local dimension four, and a basis is
\begin{equation}\label{eq:at0_basis_polynomial}
\cB_0^{\text{poly}} = \{\x-\x_{124},\x-\x_{034},\x-\x_{025},\x-\x_{024}\}.
\end{equation}
The normal flux $(\x-\x_{ijk})\cdot\nu_\ell|_{f_\ell}$ is zero if $\ell\in\{i,j,k\}$ and strictly positive
otherwise.

\subsubsection{A simple supplemental space for $r=0$}\label{sec:at0simple}
Recalling \eqref{eq:bsigma0}, we define simply
\begin{equation}\label{eq:Supp0simple}
\Su_0^{\text{simple}} = \spn\{\bsigma_{0,0}^{1,3},\bsigma_{0,0}^{3,5}\}.
\end{equation}
A local basis is
$\cB_0^{\text{simple}} = \cB_0^{\text{poly}}\cup\{\bsigma_{0,0}^{1,3},\bsigma_{0,0}^{3,5}\}$. To prove that the DOFs are independent, we compute the flux matrix, which is an ordinary matrix of numbers when $r=0$.  This matrix \Add is a permutation of $\big[\cF_{024}~~\cF_{135}\big]$, which \endAdd has the sign
\begin{equation}\label{eq:fluxMatrix0simple}
\text{signum}\big( \Add \big[\cF_{024}(\cB_0^{\text{simple}})~~\cF_{135}(\cB_0^{\text{simple}})\big]\endAdd \big) = \left[\begin{array}{ccc|ccc}
+&0&0   &0&+&+\\
0&+&0   &+&0&+\\
0&0&+   &+&+&0\\\hline
0&0&0   &+&+&+\\
0&0&0   &+&-&0\\
0&0&0   &0&+&-
\end{array}\right],
\end{equation}
where a plus or minus sign ($+$ or $-$) indicates that the number is strictly positive or negative, respectively.  
\Add
Obviously, matrix (\ref{eq:fluxMatrix0simple}) is invertible if 
the determinant of the lower right $3 \times 3$ submatrix is nonzero.
This determinant is strictly positive if we expand the $3\times 3$
matrix by Sarrus' rule.
\endAdd
Since a matrix of this form is invertible, we can decouple the DOFs \eqref{eq:flux_DOFs}; thus, the mixed finite element is well defined.

A set of shape functions can be defined by inverting $\cF(\cB_0^{\text{simple}})$.  If we let
$C=(\cF(\cB_0^{\text{simple}}))^{-1}$, then the shape function for the DOF on face~$i$ (i.e.,
$\cF(\phi_{0,i}^{\text{simple}})=\e_{i+1}^T$) is
\begin{align}\nonumber
\bphi_{0,i}^{\text{simple}}(\x) &= C_{i,1}(\x-\x_{124}) + C_{i,2}(\x-\x_{034}) + C_{i,3}(\x-\x_{025})\\
&\quad  + C_{i,4}(\x-\x_{024}) + C_{i,5}\bsigma_{0,0}^{1,3} + C_{i,6}\bsigma_{0,0}^{3,5}.\label{eq:shape0simple}
\end{align}
In fact, an explicit basis can be constructed without the need to invert a matrix.  
\Add
Recall that for any point $\x$ on face 1, $(\x-\x_{024})\cdot\nu_1$ denotes the distance
from point $\x_{024}$ to face 1, which is a constant.
\endAdd
Compute the
numbers
$$
\alpha=(\x-\x_{024})\cdot\nu_1|_{f_1},\quad 
\beta=(\x-\x_{024})\cdot\nu_3|_{f_3},\quad\text{and}\quad
\gamma=(\x-\x_{024})\cdot\nu_5|_{f_5},
$$
\Add which are positive due to the convexity of $E$, and then 
$\cF_{024}(\x-\x_{024},\bsigma_{0,0}^{1,3},\bsigma_{0,0}^{3,5})$ vanishes and \endAdd
\begin{equation}\label{eq:simple0fluxes}
\cF_{135}(\x-\x_{024},\bsigma_{0,0}^{1,3},\bsigma_{0,0}^{3,5})= \left[\begin{matrix}
\alpha&\beta&\gamma\\
1/|f_1|&-1/|f_3|&0\\
0&1/|f_3|&-1/|f_5|
\end{matrix}\right].
\end{equation}
\Add Guided by \endAdd these fluxes, we construct \Add the following linear combinations: \endAdd
\begin{align}
\label{eq:simpleShape1}
\bphi_{0,1}^{\text{simple}}(\x)
&= |f_1|\frac{\x-\x_{024} + (|f_3|\beta+|f_5|\gamma)\bsigma_{0,0}^{1,3} + |f_5|\gamma\bsigma_{0,0}^{3,5}}
              {|f_1|\alpha+|f_3|\beta+|f_5|\gamma},\\
\label{eq:simpleShape3}
\bphi_{0,3}^{\text{simple}}(\x)
&= |f_3|\frac{\x-\x_{024} - |f_1|\alpha\bsigma_{0,0}^{1,3} + |f_5|\gamma\bsigma_{0,0}^{3,5}}
              {|f_1|\alpha+|f_3|\beta+|f_5|\gamma},\\
\label{eq:simpleShape5}
\bphi_{0,5}^{\text{simple}}(\x)
&= |f_5|\frac{\x-\x_{024} - |f_1|\alpha\bsigma_{0,0}^{1,3} - (|f_1|\alpha+|f_3|\beta)\bsigma_{0,0}^{3,5}}
              {|f_1|\alpha+|f_3|\beta+|f_5|\gamma}.
\end{align}
\Add Using \eqref{eq:simple0fluxes}, inspection shows that indeed $\cF(\phi_{0,i}^{\text{simple}})=\e_{i+1}^T$, $i=1,3,5$. \endAdd From these functions, we then construct
\begin{align}
\label{eq:simpleShape0}
\bphi_{0,0}^{\text{simple}}(\x)
&= \frac{\nu_2\times\nu_4 - (\nu_2\times\nu_4)\cdot(\nu_1\bphi_{0,1}^{\text{simple}}
 + \nu_3\bphi_{0,3}^{\text{simple}} + \nu_5\bphi_{0,5}^{\text{simple}})}{(\nu_2\times\nu_4)\cdot\nu_0},\\
\label{eq:simpleShape2}
\bphi_{0,2}^{\text{simple}}(\x)
&= \frac{\nu_0\times\nu_4 - (\nu_0\times\nu_4)\cdot(\nu_1\bphi_{0,1}^{\text{simple}}
 + \nu_3\bphi_{0,3}^{\text{simple}} + \nu_5\bphi_{0,5}^{\text{simple}})}{(\nu_0\times\nu_4)\cdot\nu_2},\\
\label{eq:simpleShape4}
\bphi_{0,4}^{\text{simple}}(\x)
&= \frac{\nu_0\times\nu_2 - (\nu_0\times\nu_2)\cdot(\nu_1\bphi_{0,1}^{\text{simple}}
 + \nu_3\bphi_{0,3}^{\text{simple}} + \nu_5\bphi_{0,5}^{\text{simple}})}{(\nu_0\times\nu_2)\cdot\nu_4}.
\end{align}
\Add Using the property $\cF(\phi_{0,i}^{\text{simple}})=\e_{i+1}^T$, $i=1,3,5$, already established, a careful inspection of \eqref{eq:simpleShape0}--\eqref{eq:simpleShape4} shows that these functions also satisfy the required property $\cF(\phi_{0,i}^{\text{simple}})=\e_{i+1}^T$, $i=0,2,4$.  Therefore, we have constructed a simple set of shape functions for the lowest order case $r=0$. \endAdd

\subsubsection{A more general supplemental space for $r=0$}\label{sec:at0general}
While $\cB_0^{\text{simple}}$ is well defined and simple to implement, it is defined in a highly
non-symmetric way. One could average over all similar constructions, but it is not clear how to
weight them. An alternative is to add supplements that are as different as possible from the
polynomial part $\Po_0^3\oplus\x\Po_0$, and subject to the divergence-free constraint.  A criterion
is to consider the fluxes generated by this part, and take supplements with fluxes that span the
orthogonal complement.  We denote the flux matrix for $\cB_0^{\text{poly}}$ as
\Add
\begin{equation}\label{eq:fluxMatrix0_polynomial}
M = \left[\cF_{024}(\cB_0^{\text{poly}})\:\:\cF_{135}(\cB_0^{\text{poly}})\right] = \left[\begin{matrix}
a_1&0  &0  &0     &b_1  &c_1    \\
0  &b_2&0  &a_2   &0    &c_2    \\
0  &0  &c_3&a_3   &b_3  &0      \\
0  &0  &0  &\alpha&\beta&\gamma 
\end{matrix}\right],
\end{equation}
\endAdd
where each letter ($a_i$, $b_i$, $c_i$, $\alpha$, $\beta$, and $\gamma$) stands for a specific
positive number.  The orthogonal complement of the row space of $M$ is easily seen to be spanned by
$N^T$ \Add (i.e., rank$\,M = 4$, rank$\,N = 2$ and $MN^T = 0$), \endAdd where
\Add
\begin{equation}\label{eq:fluxMatrix0_ortho}
N = \left[\begin{matrix}
\alpha b_1/a_1&\ -\beta a_2/b_2&\ (\alpha b_3 - \beta a_3)/c_3&\ \beta&\ -\alpha&\ 0\\
(\beta c_1 - \gamma b_1)/a_1&\ \beta c_2/b_2&\ -\gamma b_3/c_3&\ 0&\ \gamma&\ -\beta
\end{matrix}\right].
\end{equation}
\endAdd
Let $S$ denote the $2\times6$ matrix with rows being the desired supplemental fluxes.  The
divergence-free constraint can be written as $S\bvarphi=0$ in terms of the vector of face areas, which
is
\begin{equation}\label{eq:areaVector}
\bvarphi = \big(|f_0|,|f_1|,|f_2|,|f_3|,|f_4|,|f_5|).
\end{equation}
We define $S$ to be the projection of $N$ to the orthogonal complement of $\spn\{\bvarphi\}$, i.e.,
\begin{equation}\label{eq:desiredS0}
S = N\Big(I - \frac{\bvarphi\bvarphi^T}{\bvarphi^T\bvarphi}\Big),
\end{equation}
and then we define $\Su_0=\spn\{\bsigma_{0,0}^1,\bsigma_{0,0}^2\}$, where
\begin{align}
\label{eq:at0general1}
\bsigma_{0,0}^1 &= \cP_E\big(S_{1,1}\hatbsigma_{0,0}^0 + S_{1,2}\hatbsigma_{0,0}^1 + S_{1,3}\hatbsigma_{0,0}^2
     + S_{1,4}\hatbsigma_{0,0}^3 + S_{1,5}\hatbsigma_{0,0}^4 + S_{1,6}\hatbsigma_{0,0}^5\big),\\
\label{eq:at0general2}
\bsigma_{0,0}^2 &= \cP_E\big(S_{2,1}\hatbsigma_{0,0}^0 + S_{2,2}\hatbsigma_{0,0}^1 + S_{2,3}\hatbsigma_{0,0}^2
     + S_{2,4}\hatbsigma_{0,0}^3 + S_{2,5}\hatbsigma_{0,0}^4 + S_{2,6}\hatbsigma_{0,0}^5\big),
\end{align}
since, by \eqref{eq:bsigma0_div_flux} and \eqref{eq:Piola_div}--\eqref{eq:Piola_normal}, these
supplements satisfy the constraint of being divergence-free and produce the desired fluxes $S$ on
each face.

It remains to verify that the DOFs are independent after applying the projection.  To this end, we
note that $\bvarphi$ is not in the span of the rows of $N$.  This is true since $M\bvarphi\ne0$ (at
least one row of $M$ represents a function with a nonzero divergence), which implies that
$\bvarphi\not\in(M^T)^\perp=\text{row}(N)$.  Independence of the DOFs is a consequence of the following,
more general lemma.

\begin{lemma}\label{lem:independence_of_projection}
  Suppose that $M$ is $m\times(m+n)$, $N$ is $n\times(m+n)$, and
  $\left[\begin{matrix}M\\N\end{matrix}\right]$ is invertible.  Let $\bvarphi$ be an $(m+n)$-vector
  that does not lie in the row space of $N$.  Let the projection
  $P_\bvarphi = \dfrac{\bvarphi\bvarphi^T}{\bvarphi^T\bvarphi}$. If $S = N(I-P_\bvarphi)$, then
  $\left[\begin{matrix}M\\S\end{matrix}\right]$ is invertible.
\end{lemma}

\begin{proof}
  By a change of basis, we may assume that $M=\left[\begin{matrix}I_m&0\end{matrix}\right]$ and
  $N=\left[\begin{matrix}0&I_n\end{matrix}\right]$.  Normalize and partition
  $\bvarphi=\left(\begin{matrix}{\mathbf a}\\{\mathbf b}\end{matrix}\right)$ into $m$- and $n$-subvectors.  Now the
  projection in block form is
\begin{equation*}
P_\bvarphi = \left[\begin{matrix}{\mathbf a}{\mathbf a}^T&{\mathbf a}{\mathbf b}^T\\
{\mathbf b}{\mathbf a}^T&{\mathbf b}{\mathbf b}^T\end{matrix}\right],
\end{equation*}
and $S=\left[\begin{matrix}-{\mathbf b}{\mathbf a}^T&I_n-{\mathbf b}{\mathbf b}^T\end{matrix}\right]$.  Since $\|{\mathbf b}{\mathbf b}^T\|<1$ (recall
${\mathbf a}\ne0$), we conclude that $I_n-{\mathbf b}{\mathbf b}^T$ is invertible, and thus also
$\left[\begin{matrix}I_m&0\\-{\mathbf b}{\mathbf a}^T&I_n-{\mathbf b}{\mathbf b}^T\end{matrix}\right] =
\left[\begin{matrix}M\\S\end{matrix}\right]$.
\end{proof}

\subsection{The case $r=1$}\label{sec:at1}

We concentrate on the reduced space $\V_1^\red(E)=\Po_1^3\oplus\Su_1$, since we merely add
$\x\tilde\Po_1$ to define $\V_1(E)$.  The divergence of $\V_1^\red(E)$ is constant as in the case
$r=0$, but now the normal face fluxes are linear, so there are 18 of them in total.  Since
$\dim\Po_1^3=12$, we need 6 supplements.

Please recall the notation from Fig.~\ref{fig:geometry}. We can view the hexahedron as containing a tetrahedron nestled in the corner near $\x_{024}$, i.e., the tetrahedron with the four vertices $\x_{024}$, $\x_{124}$, $\x_{034}$, and $\x_{025}$.  The usual BDM (i.e., BDDF) space on tetrahedra \cite{BDDF_1987} is $\Po_1^3$, so we know that we can set the fluxes independently on the faces $0$, $2$, and $4$ \Add by polynomial vector functions (since these fluxes are independent degrees of freedom for the tetrahedral element $\Po_1^3\subset\V_1^\red(E)$). To find these functions, we first \endAdd define the six linear functions
\begin{equation}
  \lambda_i(\x) = -(\x - \x_i)\cdot\nu_i,\quad i=0,1,\ldots,5,
\end{equation}
and the linear function associated with the plane $f_6$ through $\x_{124}$, $\x_{034}$, and $\x_{025}$,
\begin{equation}
  \lambda_6(\x) = -(\x - \x_6)\cdot\nu_6,
\end{equation}
where $\x_6$ lies on $f_6$ and $\nu_6$ is the unit normal pointing {\em into} the tetrahedron.

Since $\grad\lambda_i=-\nu_i$, we have that
\begin{equation*}
\curl(\lambda_i\lambda_j\nu_k) = -\lambda_i\nu_j\times\nu_k - \lambda_j\nu_i\times\nu_k,
\end{equation*}
which has no normal flux on faces $i$, $j$, and $k$.  \Add As we show below, \endAdd we can
independently set the 9 fluxes on the faces 0, 2, and 4, respectively, by the functions
\begin{alignat}3
\bpsi_{0} &= \x - \x_{124},
&\quad\bpsi_{1} &= \curl(\lambda_2\lambda_6\nu_4),
&\quad\bpsi_{2} &= \curl(\lambda_4\lambda_6\nu_2),\\
\bpsi_{3} &= \x - \x_{034},
&\quad\bpsi_{4} &= \curl(\lambda_0\lambda_6\nu_4),
&\quad\bpsi_{5} &= \curl(\lambda_4\lambda_6\nu_0),\\
\bpsi_{6} &= \x - \x_{025},
&\quad\bpsi_{7} &= \curl(\lambda_0\lambda_6\nu_2),
&\quad\bpsi_{8} &= \curl(\lambda_2\lambda_6\nu_0).
\end{alignat}
The rest of the polynomial space is associated to $f_6$, and consists of the functions
\begin{equation}
\bpsi_{9} = \x - \x_{024},
\quad\bpsi_{10} = \curl(\lambda_0\lambda_2\nu_4),
\quad\bpsi_{11} = \curl(\lambda_2\lambda_4\nu_0).
\end{equation}

It is convenient for the discussion to map $E$ to a simpler shape $\tilde E$ using an affine map.
In the case of an affine map, no polynomial spaces are changed, so conclusions about fluxes on
$\dd\tilde E$ hold for $\dd E$.  We take $\tilde E$ as in Fig.~\ref{fig:geometry}, but it is the
result of a translation that makes $\x_{024}=0$.  Rotations, dilations, and shear maps can then
make $\x_{124}=\e_1$, $\x_{034}=\e_2$, and $\x_{025}=\e_3$.  We proceed as if $E=\tilde E$.  Then
\begin{alignat*}4
\nu_0&=-\e_1,&\quad\nu_2&=-\e_2,&\quad\nu_4&=-\e_3,&\quad\nu_6&=-(\e_1+\e_2+\e_3)/\sqrt{3},\\
\lambda_0&=x_1,&\quad\lambda_2&=x_2,&\quad\lambda_4&=x_3,&\quad\lambda_6&=(x_1+x_2+x_3-1)/\sqrt{3}.
\end{alignat*}
Thus, for face~0,
\begin{equation}
\bpsi_{0} = \threeVec{x_1-1}{x_2}{x_3}\!,
\ \ \bpsi_{1} = \frac1{\sqrt{3}}\!\!\threeVec{1-x_1-2x_2-x_3}{x_2}{0}\!,
\ \ \bpsi_{2} = \frac1{\sqrt{3}}\!\!\threeVec{x_1+x_2+2x_3-1}{0}{-x_3}\!,
\end{equation}
and so we compute the columns of $\cF$ for faces 0, 2, and 4 as
\begin{equation}\def\emDash{\text{---}}
\cF_{024}(\bpsi_{0},\bpsi_{1},\bpsi_{2}) = \left[\begin{matrix}
1&\ 0&\ 0\\
2x_2+x_3-1&\ 0&\ 0\\
1-x_2-2x_3&\ 0&\ 0
\end{matrix}\right].
\end{equation}
The other two triples, $(\bpsi_{3}, \bpsi_{4}, \bpsi_{5})$ for face~2 and
$(\bpsi_{6}, \bpsi_{7}, \bpsi_{8})$ for face~4, are similar, so we conclude that indeed these 9
functions independently set the 9 fluxes on the faces 0, 2, and 4.

For the other three faces 1, 3, and 5, we have that
\begin{equation}
\bpsi_{9} = \threeVec{x_1}{x_2}{x_3},
\quad\bpsi_{10} = \threeVec{-x_1}{x_2}{0},
\quad\bpsi_{11} = \threeVec{0}{-x_2}{x_3}.
\end{equation}
Note that these three functions have no normal flux on faces 0, 2, and 4.  
\Add
In the following discussion, for simplicity, we replace $\bpsi_{9}$, $\bpsi_{10}$, and $\bpsi_{11}$
with $\bpsi_{9}^*$, $\bpsi_{10}^*$, and $\bpsi_{11}^*$ where
\begin{equation}
\bpsi_{9}^* = \threeVec{x_1}{0}{0},
\quad\bpsi_{10}^* = \threeVec{0}{x_2}{0},
\quad\bpsi_{11}^* = \threeVec{0}{0}{x_3}.
\end{equation}
We can do this because
\begin{equation}
\frac13\left(
\begin{matrix}
x_1 & -x_1 & 0\\
x_2 & x_2 & -x_2\\
x_3 & 0 & x_3
\end{matrix}
\right)\left(
\begin{matrix}
1 & 1 & 1\\
-2 & 1 & 1\\
-1 & -1 & 2
\end{matrix}
\right) = \left(
\begin{matrix}
x_1 & 0 & 0\\
0 & x_2 & 0\\
0 & 0 & x_3
\end{matrix}
\right),
\end{equation}
and the transformation matrix is invertible, so
$\bpsi_{9}$, $\bpsi_{10}$, and $\bpsi_{11}$ span the same space as
$\bpsi_{9}^*$, $\bpsi_{10}^*$, and $\bpsi_{11}^*$.
Therefore,
\begin{align}\label{eq:flux9-11}
\cF_{135}(\bpsi_{9}^*,\bpsi_{10}^*,\bpsi_{11}^*)
= \left[\begin{matrix}
x_1\nu_{1,1}\ &\ x_1\nu_{3,1}\ &\ x_1\nu_{5,1}\\
x_2\nu_{1,2}&x_2\nu_{3,2}&x_2\nu_{5,2}\\
x_3\nu_{1,3}&x_3\nu_{3,3}&x_3\nu_{5,3}
\end{matrix}\right].
\end{align}
\endAdd
We must add supplements \Add to the set $\{\bpsi_{9}^*,\bpsi_{10}^*,\bpsi_{11}^*\}$ that also have
no normal flux on faces 0, 2, and 4.  Moreover, the normal fluxes of the supplements on the
remaining three faces, when combined with \eqref{eq:flux9-11}, must  independently span the spaces
of linear polynomials. There are at least two ways to choose the supplements, a non-symmetric way and a symmetric way. 

\begin{theorem}[Non-Symmetric supplements]\label{thm:simpleAT1fluxes-nonsymm}
There exist constants $s$ and $t$ such that if the supplemental functions
$\bsigma_0$ to $\bsigma_3$, $\bsigma_4^*$, and $\bsigma_5^*$ are defined to take the fluxes
\begin{equation}
\label{eq:simpleAT1fluxes-nonsymm}
\cF_{135}(\bsigma_0,\bsigma_1,\bsigma_2,\bsigma_3,\bsigma_4^*,\bsigma_5^*)
= \left[\begin{matrix}
x_2-c_{2}^1&0&0\\
x_3-c_{3}^1&0&0\\
0&x_1-c_{1}^3&0\\
0&x_3-c_{3}^3&0\\
(-|f_5|c_1^5+t)/|f_1|& -t/|f_3| &x_1\\
-s/|f_1|&(-|f_5|c_2^5+s)/|f_3|&x_2
\end{matrix}\right],
\end{equation}
where the constant $c_\ell^i$ is the average over face~$i$ of the variable~$x_\ell$,
then they provide independent flux degrees of freedom.
\end{theorem}

\begin{theorem}[Symmetric supplements]\label{thm:simpleAT1fluxes-symm}
Let the supplemental functions $\bsigma_0$ to $\bsigma_5$ take the fluxes
\begin{equation}
\label{eq:simpleAT1fluxes-symm}
\cF_{135}(\bsigma_0,\bsigma_1,\bsigma_2,\bsigma_3,\bsigma_4,\bsigma_5)
= \left[\begin{matrix}
x_2-c_{2}^1&0&0\\
x_3-c_{3}^1&0&0\\
0&x_1-c_{1}^3&0\\
0&x_3-c_{3}^3&0\\
0&0&x_1-c_{1}^5\\
0&0&x_2-c_{2}^5
\end{matrix}\right],
\end{equation}
where the constant $c_\ell^i$ is the average over face~$i$ of the variable~$x_\ell$.
These provide independent flux degree of freedoms as long as the matrix
\begin{equation}\label{eq:cnu}
\C\circ\hh = \left[\begin{matrix}
c_1^1\nu_{1,1}\ &\ c_1^3\nu_{3,1}\ &\ c_1^5\nu_{5,1}\\
c_2^1\nu_{1,2}&c_2^3\nu_{3,2}&c_2^5\nu_{5,2}\\
c_3^1\nu_{1,3}&c_3^3\nu_{3,3}&c_3^5\nu_{5,3}
\end{matrix}\right]
\end{equation}
is invertible.
\end{theorem}

The proofs of Theorems~\ref{thm:simpleAT1fluxes-nonsymm}
and~\ref{thm:simpleAT1fluxes-symm} appear in
Appendices~\ref{sec:appc} and~\ref{sec:appb}, respectively.
The invertibility of matrix $\C\circ\hh$ in \eqref{eq:cnu} is discussed in Appendix~\ref{sec:appa}.
We remark that we have not seen a perturbed hexahedron in practice that violates the
invertibility condition. In Appendix~\ref{sec:appa}, we prove the invertibility 
condition \eqref{eq:cnu}, i.e., Theorem~\ref{thm:2parallelAndPillars} below, in two special cases: hexahedra with at least one pair of faces being parallel and truncated pillars.

\begin{definition} \label{defn:pillar}
A cuboidal hexahedron $E$ is a truncated pillar if four of its twelve edges are parallel.
These four edges form the pillar. If they are extended to infinity,
the other two faces of $E$ are formed by truncating the pillar.
\end{definition}

\begin{theorem}\label{thm:2parallelAndPillars}
If $E$ is a cuboidal hexahedron that either has two pair of faces
being parallel or is a truncated pillar, then \eqref{eq:cnu} holds.
\end{theorem}

Meshes of cuboidal hexahedra with at least one pair of faces being parallel are
used in many applications. For any cuboidal hexahedron $E$ with flat faces, it is
easy to check this condition without transformation to $\tilde E$.
For example, the mesh $\cT_h^2$ in Section~\ref{sec:numerics} satisfies this condition.

Meshes of truncated pillars are widely used.
For example, in reservoir simulation and geological modeling, it is very common that the dataset
is given in the corner-point grid format~\cite{Ponting_1989}. 
The grid format gives a set of pillar lines which
run from the top to the bottom of the model and, in many cases, the lines are vertical.
The mesh $\cT_h^3$ in Section~\ref{sec:numerics} is an example of a grid
made by truncated vertical pillars. 

The vector functions providing the fluxes we require in (\ref{eq:simpleAT1fluxes-nonsymm})
and (\ref{eq:simpleAT1fluxes-symm}) can be easily obtained using the functions
$\bsigma_{\ell,m}^1$ (\ref{eq:desired_bsigma}) and $\bsigma_{0,0}^{i,j}$ (\ref{eq:bsigma0})
defined in Section~\ref{sec:supplements}. For example, $\bsigma_0$ here is exactly 
$\bsigma_{1,0}^1$ of (\ref{eq:desired_bsigma}), and
\begin{align*}
\cF_{135}(\bsigma_4^*) &= \left[(-|f_5|c_1^5+t)/|f_1|\quad -t/|f_3| \quad x_1\right]\\
&= \left[0 \quad 0 \quad x_1-c_1^5\right]
+ |f_5|c_1^5\left[-\frac{1}{|f_1|} \quad 0 \quad \frac{1}{|f_5|}\right]
+ t\left[\frac{1}{|f_1|} \quad -\frac{1}{|f_3|} \quad 0\right]\\
&= \cF_{135}\left(\bsigma_{1,0}^5\right) + 
|f_5|c_1^5 \cF_{135}\left(\bsigma_{0,0}^{5,1}\right)
+ t\cF_{135}\left(\bsigma_{0,0}^{1,3}\right),
\end{align*}
so $\bsigma_4^* = \bsigma_{1,0}^5 + |f_5|c_1^5\bsigma_{0,0}^{5,1} + t\bsigma_{0,0}^{1,3}\myStrut$.

In conclusion, if we know that $\C\circ\hh$ is invertible for the meshes used, 
we can apply the symmetric supplements. 
On the other hand, one can always take the non-symmetric supplements 
for any mesh, provided $s$ and $t$ are chosen properly.
\endAdd
A general method for handling $r=1$ is contained in the next subsection.

\subsection{The general case $r\ge1$}\label{sec:at}

In general, the DOFs of our mixed finite element spaces are allocated as 
\begin{equation}
\begin{alignedat}2
\V_r(E) &= \Po_r^3\oplus\x\tilde\Po_r\oplus\Su_r &&= \Ed_r\oplus\Di_r\oplus\Bu_r\\
\text{or}\quad
\V_r^{\red}(E) &= \Po_r^3\oplus\Su_r &&= \Ed_r\oplus\Di_r^{\red}\oplus\Bu_r.
\end{alignedat}
\end{equation}
Here $\Ed_r$ are the functions that have constant divergence and independently cover the normal flux
DOFs \eqref{eq:flux_DOFs}.  The functions in $\Di_r$ or $\Di_r^{\red}$ match the (nonconstant)
divergence DOFs \eqref{eq:div_DOFs}. One of these functions can be constructed from a basis function
in $\x\tilde\Po_r^*$ or $\x\tilde\Po_{r-1}^*$, respectively, but then modified by the functions in
$\Ed_r$ to remove the face normal fluxes.  Finally, the divergence-free bubbles $\Bu_r$ are left
over, and provide the final set of DOFs.  Since $\Po_r^3=\Curl\Po_{r+1}^3\oplus\x\Po_{r-1}$, we
conclude that
\begin{equation}
\Ed_r\oplus\Bu_r=\Curl\Po_{r+1}^3\oplus\x\Po_0\oplus\Su_r.
\end{equation}
Thus, our task is to construct the supplemental space $\Su_r$ of functions with zero divergence so
that the normal flux DOFs \eqref{eq:flux_DOFs} in $\Curl\Po_{r+1}^3\oplus\x\Po_0\oplus\Su_r$ are
independent.

Cockburn and Fu~\cite{Cockburn_Fu_2017_mDecompIII} determined the minimal number of supplemental
functions (which they call ``filling functions'') needed to produce the space $\Ed_r$ on various
elements, including a cuboidal hexahedron.  In particular, \cite[Lemma~4.6 and
Theorems~2.10--2.15]{Cockburn_Fu_2017_mDecompIII} identify the fluxes required (but note that they
label the faces counting from 1 rather than 0).  Their construction is to obtain supplements that
have no flux on faces~0,~1, and~2.  They specify the needed fluxes on face~3, but allow any flux on
the last two faces.  They then specify the needed fluxes on face~4, but again allow any flux on the
last face.  Finally, face~5 has a set of required fluxes, and these can be matched by
divergence-free functions.  As mentioned previously, Cockburn and Fu use a mesh of tetrahedral
elements within the hexahedron to construct their supplemental functions.  We can instead use the
ideas of Sections~\ref{sec:presupplements}--\ref{sec:supplements}.

The number of additional fluxes (see~\cite[Cor.~4.5 and Table~4]{Cockburn_Fu_2017_mDecompIII}) is
bounded by $3(r+1)$ and depends on the geometry, in particular, on the number of parallel faces.
The cube requires $3(r+1)$ supplemental functions.  It is numerically delicate to vary the number of
supplemental functions based on the number of parallel sides, since an element $E$ may have almost,
but not quite, parallel faces.

A numerically safe way to proceed is to use the general construction of
Subsection~\ref{sec:at0general}.  Since it is difficult to characterize what functions lie in
$\Bu_r$ (see, however, \cite{Cockburn_Fu_2017_mDecompIII}), we simply compute the flux matrix of the
entire polynomial part of the space, i.e., of a basis for $\Curl\Po_{r+1}^3\oplus\x\Po_0$, which has
dimension $n=\dim\Po_r^3-\dim(\x\Po_{r-1}^*)=\frac16(r+2)(r+1)(2r+9)+1$.  To proceed, it is convenient
to express the flux matrix as an ordinary matrix of numbers, so we expand every normal flux
polynomial in a basis that includes~1 and everything orthogonal to~1.  A simple choice is displayed
in \eqref{eq:desired_bsigma} for face~1, i.e., take~1 and the functions
$x_{i_1}^\ell x_{j_1}^m-c_{\ell,m}^1$ for $1\le\ell+m\le r$.  The expansion coefficients give the
matrix $M^{\text{full}}$, which is $n\times3(r+2)(r+1)$.

We reduce the number of rows in $M^{\text{full}}$ to $M$ by including only a basis for the row
space.  This removes the interior bubble parts of the space.  It may be better to compute the
singular values of $M^{\text{full}}$ and remove all rows corresponding to small singular values.  In
fact, we suggest reducing $M^{\text{full}}$ to an $n-3(r+1)$ matrix, so that $3(r+1)$ supplements are
needed, regardless of the geometry.  This may create more interior bubble functions than is
necessary, but it safely handles any geometry.

We proceed to find a basis $N^T$ of $(M^T)^\perp$. Let the area vector $\bvarphi$ be analogous to the
one defined in \eqref{eq:areaVector} (it is the same, except that it has more zeros).  The desired
supplemental fluxes $S$ are then defined by the formula in \eqref{eq:desiredS0}, i.e.,
$S = N\Big(I - \dfrac{\bvarphi\bvarphi^T}{\bvarphi^T\bvarphi}\Big)$.  We construct supplemental
functions $\Su_r$ having these fluxes.  By Lemma~\ref{lem:independence_of_projection}, these fluxes
are independent of the ones from $M$, and so the space $\Ed_r$ is well-defined.  Any extra functions
are divergence-free bubbles, which can be modified to have no face fluxes.

In the hybrid form of the mixed method~\cite{Arnold_Brezzi_1985}, the Lagrange multiplier space on
the face~$f$ is simply $\Po_r(f)$, and implementation is clear up to evaluation of the integrals
over the elements.  If the hybrid form is not used, one needs $H(\Div)$-conforming finite element
shape functions to form a local basis.  This is done by inverting the numerical counterpart of the
local flux matrix, as discussed in \eqref{eq:shape0simple} for $r=0$.

\subsection{Construction of the $\pi$ operator}\label{sec:pi_operator}

Once the spaces $\Ed_r$, $\Di_r$ or $\Di_r^{\red}$, and $\Bu_r$ have been determined, one can
define the Raviart-Thomas~\cite{Raviart_Thomas_1977} or Fortin~\cite{Brezzi_Fortin_1991} projection
operator $\pi_r$ onto $\V_r(E)=\Ed_r\oplus\Di_r\oplus\Bu_r$ or $\pi_r^{\red}$ onto
$\V_r^{\red}(E)=\Ed_r\oplus\Di_r^{\red}\oplus\Bu_r$.  One simply matches the DOFs
\eqref{eq:div_DOFs}--\eqref{eq:flux_DOFs} to fix the part in $\Ed\oplus\Di_r$ or
$\Ed\oplus\Di_r^{\red}$.  To these DOFs, we add
\begin{equation}\label{eq:bubble_DOFs}
(\vv,\bpsi)_E\quad\forall\bpsi\in\Bu_r.
\end{equation}

\Add Because of \eqref{eq:div_DOFs}, \endAdd these projection operators satisfy the commuting diagram property, namely, that
\begin{equation}\label{eq:commuting_diagram}
\div\pi_r\vv = \cP_{W_r}\div\vv
\quad\text{and}\quad
\div\pi_r^\red\vv = \cP_{W_{r-1}}\div\vv,
\end{equation}
where $\cP_{W_s}$ is the $L^2$-projection onto $W_s$.  Moreover, since our spaces contain full sets
of polynomials, $\V_r\times W_r$ will have full $H(\Div)$-approximation properties and
$\V_r^{\red}\times W_{r-1}$ will have reduced $H(\Div)$-approximation properties.  Moreover,
we have the following result.

\begin{lemma}\label{lem:convergence}
\Add Assume that the computational mesh is shape-regular. \endAdd
The spaces $\V_r\times W_r$ and $\V_r^\red\times W_{r-1}$ satisfy the inf-sup conditions
\begin{equation}
\inf_{\text{\raisebox{-2.1pt}{$w\in W_r$}}}\sup_{\vv\in\V_r}\frac{(\div\vv,w)_\Omega}{\|\vv\|_{_\V}\|w\|_W}\ge\gamma>0
\quad\text{and}\quad
\inf_{\text{\raisebox{-2.1pt}{$w\in W_{r-1}$}}}\sup_{\vv\in\V_r^\red}\frac{(\div\vv,w)_\Omega}{\|\vv\|_{_\V}\|w\|_W}\ge\gamma>0.
\end{equation}
Moreover, if $\uu$ is sufficiently smooth and $h$ is the diameter of the computational mesh, then
\begin{alignat}2
\|\uu - \pi_r\uu\| + \|\uu - \pi_r^{\red}\uu\| &\le Ch^{s+1}\|\uu\|_{s+1},&&\quad 1\le s\le r,\\
\|\div(\uu - \pi_r\uu)\| &\le Ch^{s+1}\|\div\uu\|_{s+1},&&\quad 1\le s\le r,\\
\|\div(\uu - \pi_r^{\red}\uu)\| &\le Ch^{s+1}\|\div\uu\|_{s+1},&&\quad 1\le s\le r-1. 
\end{alignat}
\end{lemma}

\Add
The condition for a computational mesh to be shape regular is that each element $E$ is uniformly shape-regular \cite [pp.~104--105]{Girault_Raviart_1986}, which means that $E$ contains fifteen (overlapping) simplices constructed from any choice of four vertices, and each such simplex has an inscribed ball, the minimal radius of which is $\rho_E$. If $h_E$ denotes the diameter of $E$, the requirement is that the ratio $\rho_E/h_E\geq\sigma_*>0$, where $\sigma_*$ is independent of the meshes as $h\to0$ ($h = \max_E h_E$).

The proof of Lemma~\ref{lem:convergence} is quite standard and classic in the mixed finite element literature~(e.g., see \cite{Raviart_Thomas_1977,BDDF_1987,Brezzi_Fortin_1991}, or see the proof outlined in~\cite[Section~2]{Arbogast_Correa_2016} for the two families of similar elements defined on quadrilaterals). \endAdd


\section{Some numerical results}\label{sec:numerics}

In this section we present convergence studies for various low order mixed spaces.  We include the
new full and reduced spaces defined in Section~\ref{sec:atSpaces}, which we will designate as AT
spaces to avoid confusion.  The AT$_0$ space used is the simple one given in \eqref{eq:Supp0simple}
(or, equivalently, \eqref{eq:simpleShape1}--\eqref{eq:simpleShape4}).  The AT$_1$ full and reduced
spaces used are constructed using \Add
the symmetric supplemental fluxes of \eqref{eq:simpleAT1fluxes-symm} since
the invertibility of $\C\circ\hh$ is known for $\cT_h^2$ and $\cT_h^3$ 
(see Theorem~\ref{thm:2parallelAndPillars}).
\endAdd

The performance of the AT spaces will be compared to RT, BDDF, and ABF spaces.  For the 3-D ABF
space, we use the optimal space $\hat{\cP}^{\text{opt}}_r(\hat{K})$ of Bergot and
Durufle~\cite{Bergot_Durufle_2013}.  The test problem is defined on the unit cube $\Omega = [0,1]^3$
with the coefficient $a=1$ and the source function
$f(\x) = 3\pi^2\cos(\pi x_1)\cos(\pi x_2)\cos(\pi x_3)$.  The exact solution is
\begin{align}
p(x_1,x_2,x_3) &= \cos(\pi x_1) \cos(\pi x_2) \cos(\pi x_3),\\
\uu(x_1,x_2,x_3) &= \pi \threeVec
{\sin(\pi x_1)\cos(\pi x_2)\cos(\pi x_3)}
{\cos(\pi x_1)\sin(\pi x_2)\cos(\pi x_3)}
{\cos(\pi x_1)\cos(\pi x_2)\sin(\pi x_3)}.
\end{align}

In the computations, we apply the hybrid form of the mixed finite element
method~\cite{Arnold_Brezzi_1985}.  Let $\cT_h$ be the finite element partition of the domain
$\Omega$.  For the mixed spaces $\V_h\times W_h$, let $\V_h^*$ agree with $\V_h$ on each element
$E\in\cT_h$, but relax the condition that the normal flux be continuous on the faces of the elements.
The hybrid method is: Find $\uu_h\in\V_h^*$, $p_h\in W_h$, and $\hat p_h\in M_h$ such that
\begin{alignat}3
\label{eq:hybrid1}
&(a^{-1}\uu_h, \vv)_{E} - (p_h, \div \vv)_{E} +(\hat{p}_h, \vv\cdot \nu_i)_{\dd E} &&= 0 &&\quad\forall\vv\in\V_{h}(E), E\in\cT_h,\\
\label{eq:hybrid2}
&\sum_{E\in\cT_h} (\div \uu_h, w)_{E} &&= (f,w) _{\Omega}&&\quad \forall w\in W_{h},\\
\label{eq:hybrid3}
&\sum_{E\in\cT_h} (\uu_h\cdot\nu,\mu)_{\dd E\setminus\dd\Omega} &&=0 &&\quad \forall\mu\in M_h.
\end{alignat}
The Lagrange multiplier or trace finite element space $M_h$ is defined locally by
$M_h|_f=M_h(f)=\V_h\cdot\nu|_{f}$ for each face $f$ of the computational mesh. For the AT spaces,
$M_r(f)=\Po_r(f)$. We require that the $L^2$-projection of the Dirichlet boundary condition be
imposed on~$\hat{p}_h$.

\begin{figure}
\centerline{\parbox{.2\linewidth}{\includegraphics[width=\linewidth]{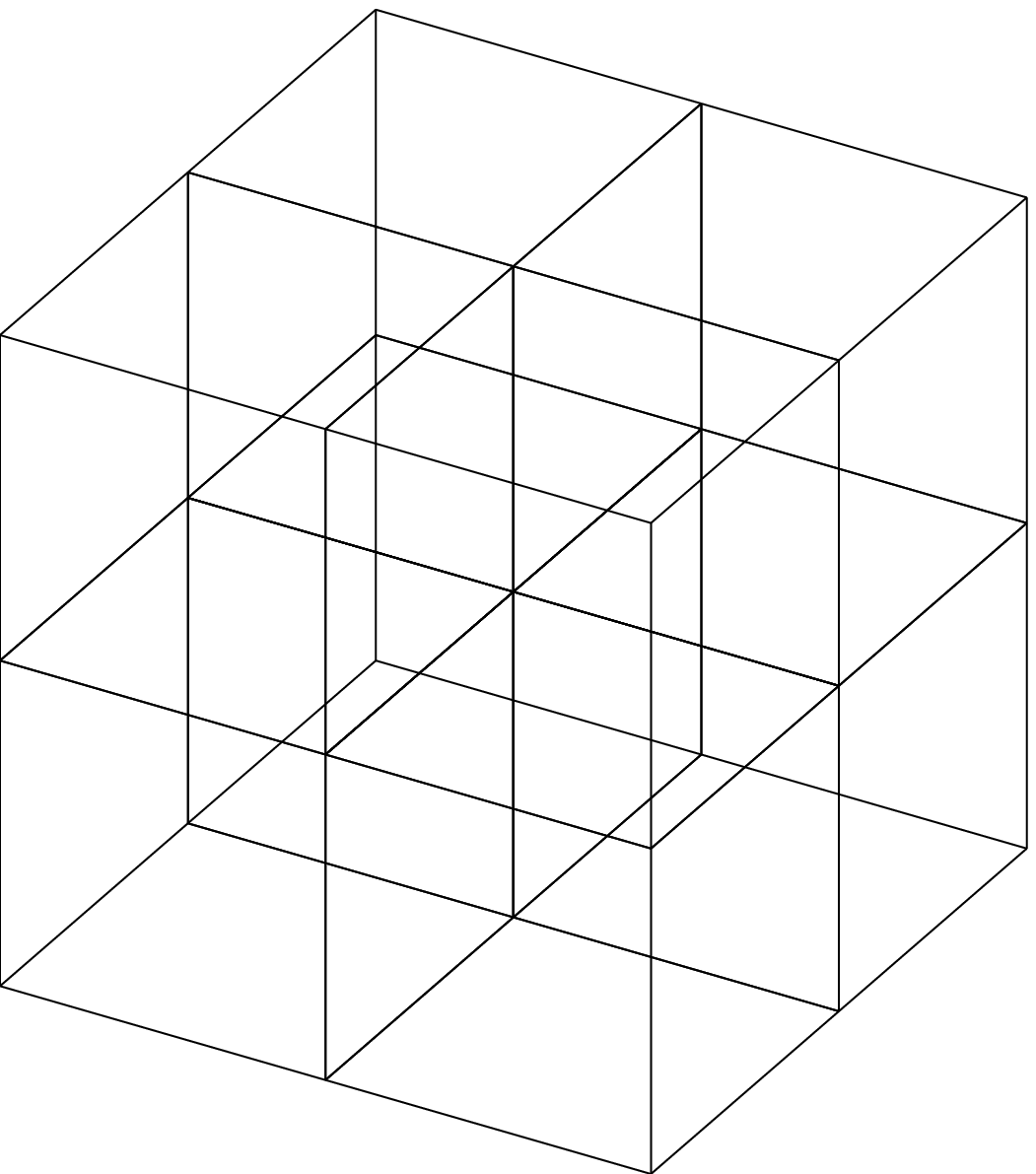}{$\cT_h^1$}}
\qquad\quad\parbox{.2\linewidth}{\includegraphics[width=\linewidth]{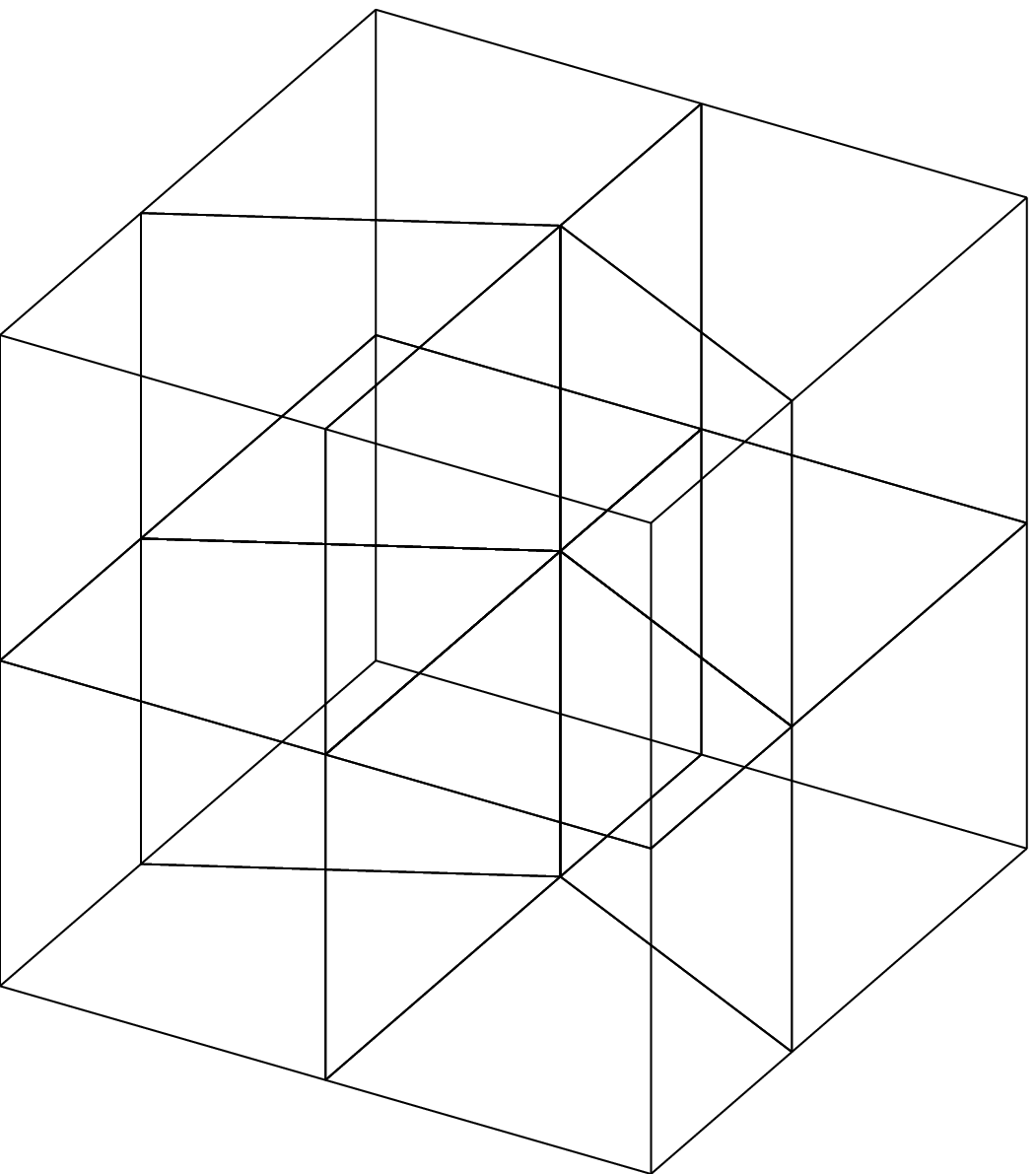}{$\cT_h^2$}}
\qquad\quad\parbox{.2\linewidth}{\includegraphics[width=\linewidth]{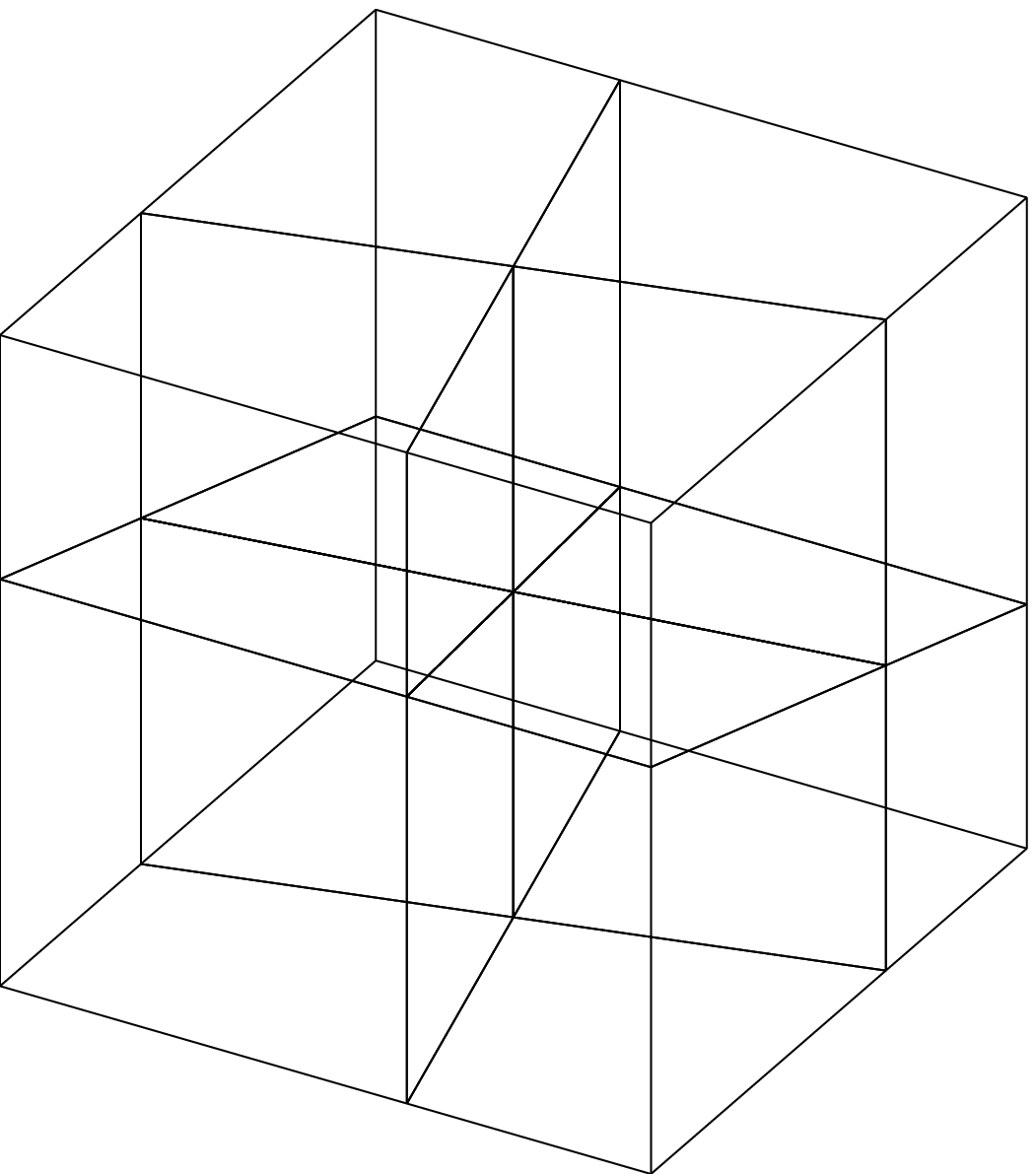}{$\cT_h^3$}}
}
\caption{Mesh of $2\times2\times2$ cubes for the three base meshes. Finer meshes are constructed by
  repeating this base mesh pattern over the domain, appropriately reflected to maintain mesh
  conformity.  Note that the meshes have 3, 2, and 0 pairs of parallel faces per element,
  respectively.}\label{fig:meshes}
\end{figure}

Solutions are computed on three different sequences of meshes.  The first sequence, $\cT_h^1$, is a
uniform mesh of~$n^3$ cubes (three sets of parallel faces per element).  The second sequence,
$\cT_h^2$, is obtained from the 2-D trapezoidal meshes used in Arnold, Boffi, and
Falk~\cite{ABF_2005} by simply lifting them in the third direction.  These elements have two pair of
parallel faces per element. The third sequence of meshes, $\cT_h^3$, is chosen so as to have no pair
of faces being parallel.  The first $2\times2\times2$ mesh for each sequence is shown in
Fig.~\ref{fig:meshes}.  Finer meshes are constructed by repeating this sub-mesh pattern over the
domain, appropriately reflected to maintain mesh conformity.

The cubical mesh $\cT_h^1$ provides a reference on which all the mixed methods work well.  It turns
out that the second and third meshes provide similar results, so we show only results for the most
irregular case of the third mesh~$\cT_h^3$.


\subsection{Full $H(\Div)$-approximation spaces}

\begin{table}
  \caption{A comparison of the dimensions of the local RT, ABF, and AT spaces
    on a hexahedron~$E$. Only the ABF and AT spaces give optimal
    order convergence on hexahedra.\label{tab:comparison_full}}
\centerline{\begin{tabular}{@{\extracolsep{0pt}}c@{\hskip2pt}|ccc}
&RT$_r$&ABF$_r$&AT$_r$\\[2pt]
\hline\\[-10pt]
dim$\V(E)$&$3(r+2)(r+1)^2$&$3(r+4)(r+2)^2$&$\frac{1}{2}(r+1)(r+2)(r+4)$\\
&&&$+\,3(r+1)$ ($+\,2$ if $r=0$)\\[2pt]
dim$W(E)$&$(r+1)^3$&$(r+2)^3+3(r+2)^2$&$\frac16(r+1)(r+2)(r+3)$\\[4pt]
\hline\\[-10pt]
$r=0$&\06 + 1&\048 + 20&\06 + 1\\
$r=1$&36 + 8&135 + 54 & 21 + 4
\end{tabular}}
\end{table}

The local number of DOFs for each full $H(\Div)$-approximation finite element space can be found in
Table~\ref{tab:comparison_full}.  Note that according to Bergot and Durufle
\cite{Bergot_Durufle_2013}, the optimal ABF$_0$ space should satisfy the property
$\cP_E(\hat{\V}^0_{\ABF}(E))\supset\Po_0^3\oplus\x\Po_0$, and so it is defined to be
$\hat{\V}^0_{\ABF}(E)=\Po_{3,1,1}\times\Po_{1,3,1}\times\Po_{1,1,3}$.  Since we solve the linear
system \eqref{eq:hybrid1}--\eqref{eq:hybrid3} using a Schur complement for $\hat{p}_h$, we will
report in this section the size of the Schur complement matrix, i.e.,~$\dim M_r$, rather than the
size of $\dim(\V_r\times W_r)$.

\begin{table}[ht]
\caption{Errors and orders of convergence for low order RT, AT, and ABF spaces on cubical meshes.}\label{tab:fullT1}
\vspace*{-4pt}
\begin{center}
\begin{tabular}{cc|c|cc|cc|cc}
\hline
& & $M_r$ & \multicolumn{2}{c|}{$\|p-p_h\|$} &
\multicolumn{2}{c|}{\quad$\|\uu-\uu_h\|$\quad } &
\multicolumn{2}{c}{$\|\nabla \cdot (\uu - \uu_h)\|$} \\
$n$ & $n^3$ & DOFs & error & order & error & order& error & order \\
\hline
\multicolumn{9}{c}{$\myStrut$RT$_0=\,\,$AT$_0$ on $\cT_h^1$ meshes}\\
\hline
\02&\0\0\0\08 & \0\0\036 & 2.417e-1 &         & 1.136e-0 &         & 7.156e-0 &         \\
\06&\0\0216  & \0\0756  & 9.110e-2 & 0.95  & 4.078e-1 & 0.97 & 2.697e-0 & 0.95 \\
12&\01728  & \05616      & 4.609e-2 & 0.99  & 2.052e-1 & 0.99 & 1.365e-0 & 0.99 \\
24&13824  & 43200          & 2.312e-2 & 1.00  & 1.027e-1 & 1.00 & 6.844e-1 & 1.00 \\
\hline
\multicolumn{9}{c}{$\myStrut$ABF$_0$ on $\cT_h^1$ meshes}\\
\hline
\02&\0\0\0\08  &   \0\0\0144 & 1.035e-2 &          & 2.523e-1  &         &  2.578e-1 &        \\ 
\06&\0\0216     & \0\03024    & 2.961e-4 & 3.17  & 2.786e-2 & 2.01 & 8.389e-3 & 3.06 \\
12&\01728      &  \022464      & 3.523e-5 & 3.05  &  6.953e-3 & 2.00 & 1.031e-3 & 3.02 \\ 
24&13824        & 172800        & 4.345e-6 & 3.01 & 1.737e-3 & 2.00 & 1.283e-4 & 3.00\\ 
\hline
\multicolumn{9}{c}{$\myStrut$RT$_1$ on $\cT_h^1$ meshes}\\
\hline
\02& \0\0\0\08  & \0\0\0144  & 5.419e-2  &         & 2.440e-1 &         & 1.603e-0 &      \\ 
\06& \0\0216     &  \0\03024    & 6.231e-3 & 1.99 & 2.773e-2 & 1.99 & 1.845e-1 & 1.99 \\
12& \01728       &  \022464     & 1.562e-3 & 2.00 & 6.945e-3 & 2.00 & 4.626e-2 & 2.00 \\
24&13824         & 172800        & 3.909e-4 & 2.00 & 1.737e-3 & 2.00 & 1.157e-2 & 2.00 \\
\hline
\multicolumn{9}{c}{$\myStrut$AT$_1$ on $\cT_h^1$ meshes}\\
\hline
\02&\0\0\0\08  &  \0\0\0108  & 1.171e-1  &         & 4.358e-1  &         & 3.465e-0    &    \\
\06&\0\0216     & \0\02268     & 1.505e-2  & 1.94 & 5.164e-2 & 1.98 & 4.455e-1 & 1.94 \\
12& \01728     & \016848       & 3.814e-3  & 1.99 & 1.298e-2 & 1.99 & 1.129e-1 & 1.99 \\
24&13824        & 129600       & 9.567e-4 & 2.00 & 3.249e-3 & 2.00 & 2.833e-2 & 2.00 \\
\hline
\end{tabular}
\end{center}
\end{table}

In Tables~\ref{tab:fullT1}--\ref{tab:fullT3}, we present the errors and the orders of the
convergence for the lowest two indices of the full $H(\Div)$-approximation spaces RT, AT, and ABF;
although, we omit ABF$_1$ because the sheer size of its linear system is computationally
excessive. On cubical meshes $\cT_h^1$, RT$_0$ and AT$_0$ coincide. Table~\ref{tab:fullT1} shows
first order approximation of the scalar $p$, the vector $\uu$, and the divergence $\div\uu$, as we
should expect. The ABF$_0$ space gives higher order approximation of all three variables on cubes
because it is constructed with higher order polynomials and, in fact, includes RT$_1$.  The results
for RT$_1$ and AT$_1$ (which are different spaces even on cubical meshes) show second order
convergence for all the variables.  The errors for RT$_1$ are smaller than AT$_1$, but RT$_1$ uses
more degrees of freedom, both locally and globally.

Table~\ref{tab:fullT3} shows that for the hexahedral mesh sequence $\cT_h^3$, RT$_0$ retains first
order convergence of the scalar but loses convergence of the vector and divergence, while AT$_0$
shows first order convergence for all three quantities. The ABF$_0$ space still gives a higher order
convergence rate for the scalar on the meshes tested.  However, we can observe that the vector and
divergence approximations quickly decrease to first order.  We also observe that AT$_1$ gives the
optimal second order approximation of all quantities, whereas RT$_1$ only retains second order for
the scalar.  The vector reduces to first order in this numerical test, but the results on the
definition of ABF$_0$ \cite{Bergot_Durufle_2013} show that this first order convergence cannot be
ensured on general meshes.  The divergence appears to be converging at less than first order.

\begin{table}[ht]
\caption{Errors and orders of convergence for low order RT, AT, and ABF spaces on $\cT_h^3$ meshes.}\label{tab:fullT3}
\vspace*{-4pt}
\begin{center}
\begin{tabular}{cc|c|cc|cc|cc}
\hline
& & $M_r$ & \multicolumn{2}{c|}{$\|p-p_h\|$} &
\multicolumn{2}{c|}{\quad$\|\uu-\uu_h\|$\quad } &
\multicolumn{2}{c}{$\|\nabla \cdot (\uu - \uu_h)\|$} \\
$n$ & $n^3$ & DOFs & error & order & error & order& error & order \\
\hline
\multicolumn{9}{c}{$\myStrut$RT$_0$ on $\cT_h^3$ meshes}\\
\hline
\02 & \0\0\0\08  &  \0\0\036 & 2.660e-1 &  & 1.185e-0 & & 7.488e-0  &\\
\06 & \0\0216 &  \0\0756 & 9.464e-2 & 0.94 & 4.591e-1 & 0.86 & 3.149e-0 & 0.76 \\
12 & \01728 & \05616 & 4.782e-2 & 0.99 & 2.630e-1 & 0.75 & 1.952e-0 & 0.60 \\
24 & 13824 & 43200 & 2.400e-2 & 1.00 & 1.838e-1 & 0.45 & 1.530e-0 & 0.29 \\
\hline
\multicolumn{9}{c}{$\myStrut$AT$_0$ on $\cT_h^3$ meshes}\\
\hline
\02 & \0\0\0\08 &   \0\0\036 & 2.661e-1   & & 1.226e-0   & & 7.873e-0    &\\
\06 & \0\0216 &   \0\0756 & 9.452e-2 & 0.94 & 4.275e-1 & 0.96 & 2.798e-0 & 0.94 \\
12 & \01728 &  \05616 & 4.771e-2 & 0.99 & 2.150e-1 & 0.99 & 1.413e-0 & 0.99 \\
24 & 13824 & 43200 & 2.394e-2 & 1.00 & 1.077e-1 & 1.00 & 7.087e-1 & 1.00 \\
\hline
\multicolumn{9}{c}{$\myStrut$ABF$_0$ on $\cT_h^3$ meshes}\\
\hline
\02 & \0\0\0\08 &   \0\0\0144 & 1.474e-2   & & 2.815e-1   & & 3.649e-1    &\\
\06 & \0\0216 &  \0\03024 & 4.706e-4 & 3.04 & 3.697e-2 & 1.85 & 2.222e-2 & 2.33 \\
12 & \01728 &  \022464 & 6.438e-5 & 2.85 & 1.310e-2 & 1.47 & 5.909e-3 & 1.77 \\
24 & 13824 & 172800 & 9.937e-6 & 2.65 & 5.537e-3 & 1.19 & 2.261e-3 & 1.30 \\
\hline
\multicolumn{9}{c}{$\myStrut$RT$_1$ on $\cT_h^3$ meshes}\\
\hline
\02 & \0\0\0\08 &   \0\0\0144 & 5.644e-2   & & 2.754e-1   & & 1.996e-0    &\\
\06 & \0\0216 &   \0\03024 & 7.098e-3 & 2.03 & 3.688e-2 & 1.83 & 2.834e-1 & 1.69 \\
12 & \01728 &  \022464 & 1.814e-3 & 2.00 & 1.311e-2 & 1.47 & 1.239e-1 & 1.15 \\
24 & 13824 & 172800 & 4.541e-4 & 2.00 & 5.547e-3 & 1.19 & 7.382e-2 & 0.64 \\
\hline
\multicolumn{9}{c}{$\myStrut$AT$_1$ on $\cT_h^3$ meshes}\\
\hline
\02 & \0\0\0\08 &   \0\0\0108 & 1.299e-1   & & 4.526e-1   & & 3.846e-0    &\\
\06 & \0\0216 &   \0\02268 & 1.600e-2 & 1.95 & 5.629e-2 & 2.00 & 4.737e-1 & 1.95 \\
12 & \01728 &  \016848 & 4.091e-3 & 1.98 & 1.436e-2 & 1.99 & 1.211e-1 & 1.98 \\
24 & 13824 & 129600 & 1.027e-3 & 2.00 & 3.600e-3 & 2.00 & 3.040e-2 & 2.00 \\
\hline
\end{tabular}
\end{center}
\vspace*{12pt}
\end{table}


\subsection{Reduced $H(\Div)$-approximation spaces}

\begin{table}
\caption{
The dimensions of the local BDDF and AT\/$^\red$ spaces on a hexahedron $E$. These spaces coincide
on rectangles, and they have the same local dimension. Only the AT\/$^\red$ spaces give optimal
order convergence on hexahedra.\label{tab:comparison_reduced}}
\centerline{\begin{tabular}{@{\extracolsep{0pt}}c@{\hskip2pt}|cccc}
&BDDF,\ AT$_r^\red$\\[2pt]
\hline&&&\\[-10pt]
dim$\V(E)$&$\frac{1}{2}(r+1)(r+2)(r+3)+3(r+1)$\\[2pt]
dim$W(E)$&$\frac16r(r+1)(r+2)$\\[4pt]
\hline&&&\\[-10pt]
$r=1$&18 + 1
\end{tabular}}
\vspace*{24pt}
\end{table}

\begin{table}[ht]
\caption{Errors and orders of convergence for BDDF$_1$ and AT$^{\red}_1$.}\label{tab:reducedT1}
\vspace*{-4pt}
\begin{center}
\begin{tabular}{cc|c|cc|cc|cc}
\hline
& & $M_r$ & \multicolumn{2}{c|}{$\|p-p_h\|$} &
\multicolumn{2}{c|}{\quad$\|\uu-\uu_h\|$\quad } &
\multicolumn{2}{c}{$\|\nabla \cdot (\uu - \uu_h)\|$} \\
$n$ & $n^3$ & DOF & error & order & error & order& error & order \\
\hline
\multicolumn{9}{c}{$\myStrut$BDDF$_1=\,\,$AT$^{\red}_1$ on $\cT_h^1$ meshes}\\
\hline
\02& \0\0\0\08  & \0\0\0108 & 2.417e-1 &         & 5.611e-1  &         & 7.156e-0   & \\
\06&  \0\0216     & \0\02268 & 9.114e-2 & 0.95 & 8.601e-2 & 1.85 & 2.697e-0 & 0.95 \\
12& \01728        & \016848 & 4.610e-2 & 0.99 & 2.249e-2 & 1.95 & 1.365e-0 & 0.99 \\
24&13824         &  129600   & 2.312e-2 & 1.00 & 5.701e-3 &  1.98 & 6.844e-1 & 1.00 \\
\hline
\end{tabular}
\end{center}
\end{table}

Next we consider the reduced H($\Div$)-approximation spaces BDDF$_1$ and AT$^{\red}_1$, which
coincide on cubical meshes.  These spaces have the same local and global dimension, as shown in
Table~\ref{tab:comparison_reduced}.  The computational results appear in
Tables~\ref{tab:reducedT1}--\ref{tab:reducedT3}. As we expect, the elements give first order
approximation for the scalar $p$ and the divergence $\div\uu$ and second order convergence for the
vector $\uu$ on cubical meshes, as shown in Table~\ref{tab:reducedT1}.  On the hexahedral meshes
$\cT_h^3$, Table~\ref{tab:reducedT3} shows that BDDF$_1$ has first order approximation of the scalar
but loses convergence of the vector and the divergence. When AT$^{\red}_1$ is used instead, the
optimal convergence rates of the cubical meshes are recovered for the hexahedral meshes, i.e.,
second order approximation for the vector~$\uu$ and first order for the scalar~$p$ and the
divergence~$\div\uu$.

\begin{table}[ht]
\caption{Errors and orders of convergence for BDDF$_1$ and AT$^{\red}_1$.}\label{tab:reducedT3}
\vspace*{-4pt}
\begin{center}
\begin{tabular}{cc|c|cc|cc|cc}
\hline
& & $M_r$ & \multicolumn{2}{c|}{$\|p-p_h\|$} &
\multicolumn{2}{c|}{\quad$\|\uu-\uu_h\|$\quad } &
\multicolumn{2}{c}{$\|\nabla \cdot (\uu - \uu_h)\|$} \\
$n$ & $n^3$ & DOF & error & order & error & order& error & order \\
\hline
\multicolumn{9}{c}{$\myStrut$BDDF$_1$ on $\cT_h^3$ meshes}\\
\hline
\02 & \0\0\0\08 &    108 & 2.665e-1 &  & 6.450e-1 & & 7.487e-0  &\\
\06 & \0\0216 &   2268 & 9.481e-2 & 0.94 & 1.164e-1 & 1.52 & 3.149e-0 & 0.76 \\
12 & \01728 &  16848 & 4.786e-2 & 0.99 & 4.000e-2 & 1.43 & 1.952e-0 & 0.60 \\
24 & 13824 & 129600 & 2.401e-2 & 1.00 & 1.723e-2 & 1.16 & 1.530e-0 & 0.29 \\
\hline
\multicolumn{9}{c}{$\myStrut$AT$^{\red}_1$ on $\cT_h^3$ meshes}\\
\hline
\02 & \0\0\0\08 &    108 & 2.660e-1 &  & 6.435e-1 & & 7.876e-0 & \\
\06 & \0\0216 &   2268 & 9.455e-2 & 0.94 & 9.760e-2 & 1.76 & 2.798e-0 & 0.94 \\
12 & \01728 &  16848 & 4.772e-2 & 0.99 & 2.610e-2 & 1.91 & 1.413e-0 & 0.99 \\
24 & 13824 & 129600 & 2.394e-2 & 1.00 & 6.753e-3 & 1.96 & 7.087e-1 & 1.00 \\
\hline
\end{tabular}
\end{center}
\end{table}

\Add

\section{Conclusions}\label{sec:conc}

We generalized the two dimensional mixed finite elements of Arbogast and
Correa~\cite{Arbogast_Correa_2016} defined on quadrilaterals to three dimensional cuboidal
hexahedra.  Our construction is similar in that vector polynomials are used directly on the element.
The space of polynomials used is rich enough to give good approximation properties over the element
for both the vector variable and its divergence (as either full or reduced $H(\Div)$-approximation).
Unfortunately, the traces of the normal components of these vector polynomials onto the faces do not
independently span the full space of polynomials.  This property is needed for $H(\Div)$-conformity.
Therefore, supplemental functions are added to the space to give the full set of edge degrees of
freedom (i.e., normal fluxes).  These supplemental functions are defined on a reference element and
mapped to the hexahedron using the Piola transform.

We provided a systematic procedure for defining supplemental functions that are divergence-free and
have any prescribed polynomial normal flux in
Sections~\ref{sec:presupplements}--\ref{sec:supplements}.  This is the key contribution of this work.

We also discussed in Section~\ref{sec:atSpaces} what normal fluxes are required of the supplemental
functions to define mixed finite element spaces.  These supplemental functions are then defined
using functions from Section~\ref{sec:supplements}. When index $r=0$ (the lowest order case), we
gave two possibilities.  The simple case has shape functions defined by the explicit formulas
\eqref{eq:simpleShape1}--\eqref{eq:simpleShape4}.  The more general case for $r=0$ in
Section~\ref{sec:at0general} requires a bit of local linear algebra,
\eqref{eq:fluxMatrix0_polynomial}--\eqref{eq:desiredS0}, to determine the fluxes required of the
supplemental functions \eqref{eq:at0general1}--\eqref{eq:at0general2}.  For $r=1$, we gave three
possibilities: (1) for elements that satisfy the invertibility 
condition \eqref{eq:cnu}, such as elements with two parallel faces or that are truncated
pillars; (2) for elements with a prescribed normal flux (up to two parameters, which must be set appropriately); and (3) for
the general case of Section~\ref{sec:at}, which applies to all $r\ge1$.  The general case requires
some local linear algebra to determine the fluxes required of the supplemental functions.

Numerical results in Section~\ref{sec:numerics} verified that our approach produces mixed finite
elements that achieve optimal full or reduced $H(\Div)$-approximation on quadrilateral meshes.

\endAdd
\Add
\appendix\normalsize

\section{On the invertibility of matrix $\C\circ\hh$}\label{sec:appa}
In Section~\ref{sec:at1} Theorem~\ref{thm:simpleAT1fluxes-symm}, we stated that the independence of
the degrees of freedom of our new spaces when $r=1$ with symmetric supplements reduces to the
invertibility of the matrix $\C\circ\hh$ (\ref{eq:cnu}), which is the Hadamard product of the
centroid matrix $\C$ (see (\ref{eq:cmatrix})) and the normal matrix $\hh$ (see (\ref{eq:nu})) for
faces $f_1$, $f_3$, and $f_5$.  In this section, we discuss the properties of these matrices and how
they relate to the geometry of the convex hexahedron $\tilde E$. We then prove the invertibility of
$\C\circ\hh$ in two special cases.

\subsection{The face normal matrix $\hh$}
Following the discussion in Section~\ref{sec:at1}, we know that
any convex cuboidal hexahedron can be affinely mapped to a simpler shape $\tilde E$,
for which $\nu_0 = -\e_1$, $\nu_2 = -\e_2$, $\nu_4 = -\e_3$ and
$\x_{124} = \e_1$, $\x_{034} = \e_2$, $\x_{025} = \e_3$.
Therefore, the normal fluxes $\nu_1$, $\nu_3$, $\nu_5$ fully define the geometry of $\tilde E$.
We define the face normal matrix
\begin{equation}\label{eq:nu}
\hh = 
\left[\begin{matrix}
\nu_{1,1} & \nu_{3,1} & \nu_{5,1}\\
\nu_{1,2} & \nu_{3,2} & \nu_{5,2}\\
\nu_{1,3} & \nu_{3,3} & \nu_{5,3}
\end{matrix}\right].
\end{equation}

The cross product of the normals of two 
intersecting faces is parallel to the edge of intersection.
Let $\tau_{ij} = \nu_i\times\nu_j$, where $\|\nu_i\times\nu_j\|>0$ for two intersecting faces.
For example (see Figure~\ref{fig:append}), $\tau_{31}=-\tau_{13}$ points from $\x_{135}$ to $\x_{134}$. 

\begin{figure}\centerline{{\setlength\unitlength{4.5pt}\begin{picture}(39.5,34)(-5.8,-1)
%
\put(10,19){\circle*{1}}\put(11,19){$\x_{024}$}
\put(1,10){\circle*{1}}\put(2,10){$\x_{124}$}
\put(28,13){\circle*{1}}\put(29,12.5){$\x_{034}$}
\put(15,0){\circle*{1}}\put(16,-0.5){$\x_{134}$}
\put(10,31){\circle*{1}}\put(11,31){$\x_{025}$}
\put(0,20){\circle*{1}}\put(1,20){$\x_{125}$}
\put(27,27){\circle*{1}}\put(28,26.5){$\x_{035}$}
\put(14.5,13){\circle*{1}}\put(15.5,12.5){$\x_{135}$}
\put(7.5,11){\circle*{1}}\put(8.25,10.75){$\cc^{1}$}
\put(21.1,13.3){\circle*{1}}\put(22,13){$\cc^{3}$}
\put(12.9,22.8){\circle*{1}}\put(13,21){$\cc^{5}$}
\thinlines
\put(10,19){\line(0,1){12}}
\put(10,19){\line(-1,-1){9}}
\put(10,19){\line(3,-1){18}}
\linethickness{1.25pt}
\Line{0}{20}{1}{10}
\Line{0}{20}{10}{31}
\Line{15}{0}{1}{10}
\Line{15}{0}{28}{13}
\Line{27}{27}{28}{13}
\Line{27}{27}{10}{31}
\thinlines
\Line{27}{27}{14.5}{13}\put(20.75,20.1){\vector(1,1){1}}\put(20.5,18.5){$\tau_{53}$}
\Line{0}{20}{14.5}{13}\put(7.25,16.5){\vector(-2,1){1}}\put(7,14.5){$\tau_{15}$}
\Line{15}{0}{14.5}{13}\put(14.75,6.5){\vector(0,-1){1}}\put(12.25,6.5){$\tau_{31}$}
\Line{7.5}{11}{3.3}{6.8}\put(3.8,7.3){\vector(-1,-1){1}}\put(1.5,5){$\nu_1$}
\Line{21.1}{13.3}{25.4}{11.9}\put(25.1,12){\vector(3,-1){1}}\put(26,10){$\nu_3$}
\Line{12.9}{22.8}{12.9}{27}\put(12.9,27){\vector(0,1){1}}\put(10.4,27){$\nu_5$}
\put(9,7){\rotatebox{-32}{\makebox(0,0){\bf face~1}}}
\put(22,10){\rotatebox{38}{\makebox(0,0){\sl\bf face~3}}}
\put(16,27){\rotatebox{-12}{\makebox(0,0){\sl\bf face~5}}}
\thicklines
\qbezier[14](27,32)(24,32)(24,28)\qbezier[14](24,28)(24,24)(20,24)\put(20,24){\vector(-1,0){1}}\put(31,32){\makebox(0,0){face~0}}
\qbezier[24](26,4)(17,4)(17,6)\put(17,6){\vector(0,1){1}}\put(30,4){\makebox(0,0){face~4}}
\qbezier[14](1,28)(3,28)(4,25)\qbezier[14](4,25)(4.5,22)(6,22)\put(6,22){\vector(1,0){1}}\put(-2.75,28){\makebox(0,0){face~2}}
\put(8,2){\makebox(0,0){$\tilde E$}}
\end{picture}}}
\caption{The geometry of $\tilde E$.}
\label{fig:append}
\end{figure}
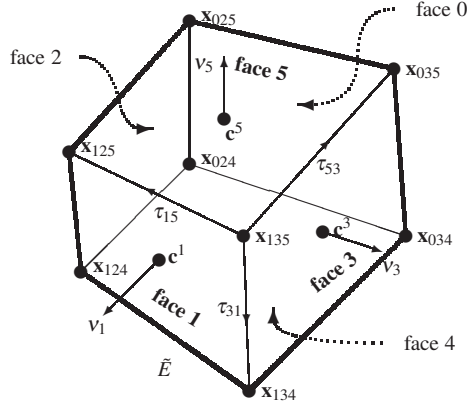

\begin{theorem}\label{thm:minors}
For any convex hexahedron $\tilde E$, all principle minors of $\hh$ are strictly positive.
\end{theorem}

\begin{proof}
We use the fact that for three vectors,
\begin{align}\nonumber
({\bf a}\times{\bf b})\times({\bf a}\times{\bf c})
= \left(({\bf b}\times{\bf c})\cdot {\bf a}\right) {\bf a}
= \det\left[{\bf a}\:\: {\bf b}\:\: {\bf c}\right] {\bf a}.
\end{align}
We first show that $\det(\hh)>0.$ 
Consider face 5 in Figure~\ref{fig:append}, for which
\begin{align}\label{eq:detH}
\tau_{53}\times\tau_{15} =
 (\nu_5 \times \nu_3) \times (\nu_1 \times \nu_5) 
 =
 (\nu_5 \times \nu_1) \times (\nu_5 \times \nu_3) 
 =  \det(\hh)\nu_5.
\end{align}
It is obvious that $(\tau_{53}\times\tau_{15})\cdot\nu_5>0$ when
face 5 is a convex quadrilateral, i.e., the triangle with vertices $\x_{135}$,
$\x_{125}$ and $\x_{035}$ does not degenerate;
therefore, $\det(\hh)>0$. 

Second, we show that the diagonal entries of $\hh$ are strictly positive.
By convexity, on face 5, $(\tau_{52}\times\tau_{05})\cdot\nu_5>0$.
Thus, computing as in (\ref{eq:detH}), we see that
$(\nu_0\times\nu_2)\cdot\nu_5>0$. Since $\nu_0\times\nu_2 = (-\e_1)\times(-\e_2)=\e_3$,
we obtain that $\nu_{5,3}>0$.
Similarly, since $(\tau_{14}\times\tau_{21})\cdot\nu_1>0$ and
$(\tau_{30}\times\tau_{43})\cdot\nu_2>0$, we have $\nu_{1,1}>0$ and $\nu_{3,2}>0$.

Finally, we show that the principal minors of order 2 are strictly positive.
By convexity, we have on face 5, $(\tau_{50}\times\tau_{35})\cdot\nu_5>0$, and so
$(\nu_3\times\nu_0)\cdot\nu_5>0$, i.e.,
\begin{align}
\det\left[
\begin{matrix}
\nu_{3,1} & -1 & \nu_{5,1}\\
\nu_{3,2} & 0 & \nu_{5,2}\\
\nu_{3,3} & 0 & \nu_{5,3}
\end{matrix}
\right] = \det\left[
\begin{matrix}
\nu_{3,2} & \nu_{5,2}\\
\nu_{3,3} & \nu_{5,3}
\end{matrix}
\right]>0.
\end{align}
The other two principal minors of order 2 can be shown from
$(\tau_{51}\times\tau_{25})\cdot\nu_5>0$
and $(\tau_{13}\times\tau_{41})\cdot\nu_1>0$.
\end{proof}

\subsection{The face centroids and matrix $\C$}\label{subsec:c}
In this section, we look at the matrix 
\begin{equation}\label{eq:cmatrix}
\C = 
\left[\begin{matrix}
c_{1}^1 \ &\ c_{1}^3 \ &\ c_{1}^5\\
c_{2}^1 & c_{2}^3 & c_{2}^5\\
c_{3}^1 & c_{3}^3 & c_{3}^5
\end{matrix}\right] = [\cc^1 \: \cc^3 \: \cc^5],
\end{equation}
where $c_\ell^i$ is the average over face $i$ of the variable $x_\ell$.
That is, $\cc^1$, $\cc^3$, $\cc^5$ are the face centroids of faces 1, 3, and 5, respectively.
Obviously, all $c_\ell^i$ are strictly positive.

Let $\cPj^3: \Re^3\rightarrow\Re^2$ denote the projection in the direction $\e_3$
to the $(x_1,x_2)-$plane. Therefore, $\cPj^3(\cc^i)$ is the centroid of the
projected face~$i$, $\cPj^3(f_i)$, $i=1$, $3$, $5$.

\begin{lemma}\label{lem:a4}
If face $2i$ and face $2i+1$, $i=0,1,2$, are parallel, then the 
determinant of the principal minor of $\C$
formed by deleting row and column $i+1$ is strictly positive.
\end{lemma}
\begin{proof}
Without loss of generality, we only need to show that when $\nu_5 = \e_3$, 
\begin{equation}\label{eq:c12}
\det\left[\begin{matrix}
c_{1}^1 & c_{1}^3 \\
c_{2}^1 & c_{2}^3  
\end{matrix}\right] > 0.
\end{equation}
When $\nu_5 = \e_3$, $\tau_{53}$ and $\tau_{34}$ are parallel,
as are $\tau_{15}$ and $\tau_{41}$.
See Figure~\ref{fig:topview2} for the projected view of $\tilde E$. 
From the figure, the area of the triangle formed by $\x_{024}$, $\cPj^3(\cc^1)$
and $\cPj^3(\cc^3)$ is positive, so
\begin{align}
\left(
\threeVec{c_1^1}{c_2^1}{0}\times
\threeVec{c_1^3}{c_2^3}{0}
\right)\cdot\e_3 >0,
\end{align}
which is (\ref{eq:c12}).
\end{proof}

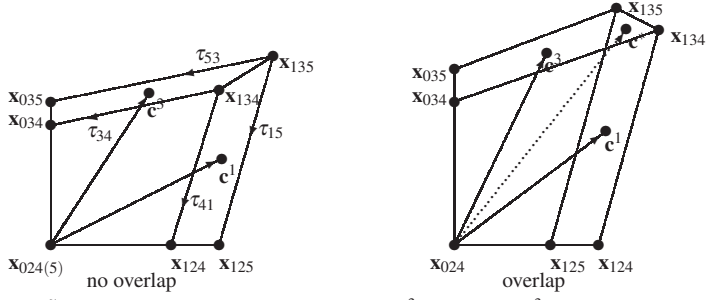
\begin{figure}\centerline{{\setlength\unitlength{4.5pt}\begin{picture}(26,24)(-4,-3)
%
\put(0,0){\circle*{.8}}\put(-3.5,-2){$\x_{024(5)}$}
\put(10,0){\circle*{.8}}\put(10,-2){$\x_{124}$}
\put(0,10){\circle*{.8}}\put(-3.5,10){$\x_{034}$}
\put(14,13){\circle*{.8}}\put(14.5,12){$\x_{134}$}
\put(14,0){\circle*{.8}}\put(14,-2){$\x_{125}$}
\put(0,12){\circle*{.8}}\put(-3.5,12){$\x_{035}$}
\put(18.5,15.8){\circle*{.8}}\put(19,14.8){$\x_{135}$}
\put(14.2,7.2){\circle*{.8}}\put(14,5.2){$\cc^1$}
\put(8.2,12.7){\circle*{.8}}\put(8,10.7){$\cc^3$}
\thinlines
\Line{0}{0}{14}{0}
\Line{0}{0}{0}{12}
\Line{10}{0}{14}{13}
\Line{0}{10}{14}{13}
\Line{14}{0}{18.5}{15.8}
\Line{0}{12}{18.5}{15.8}
\Line{14}{13}{18.5}{15.8}
\Line{0}{0}{14.2}{7.2}\put(12.8,6.5){\vector(2,1){1}}
\Line{0}{0}{8.2}{12.7}\put(7,10.9){\vector(2,3){1}}
\put(12,14.5){\vector(-4,-1){1}}\put(12,15.5){$\tau_{53}$}
\put(4,10.8){\vector(-4,-1){1}}\put(3,9){$\tau_{34}$}
\put(16.9,10){\vector(-1,-3){0.33}}\put(17.2,9.2){$\tau_{15}$}
\put(11.25,4){\vector(-1,-3){0.33}}\put(11.5,3.2){$\tau_{41}$}
\put(3,-3.5){no overlap}
\end{picture}
\qquad\qquad\begin{picture}(26,24)(-4,-3)
%
\put(0,0){\circle*{.8}}\put(-2,-2){$\x_{024}$}
\put(12,0){\circle*{.8}}\put(12,-2){$\x_{124}$}
\put(0,12){\circle*{.8}}\put(-3.5,12){$\x_{034}$}
\put(17,18){\circle*{.8}}\put(18,17){$\x_{134}$}
\put(8,0){\circle*{.8}}\put(8,-2){$\x_{125}$}
\put(0,14.7){\circle*{.8}}\put(-3.5,14){$\x_{035}$}
\put(13.5,19.8){\circle*{.8}}\put(14.5,19.5){$\x_{135}$}
\put(14.3,18.1){\circle*{.8}}\put(14.5,16.5){$\cc^*$}
\put(7.7,16.1){\circle*{.8}}\put(7.5,14.5){$\cc^3$}
\put(12.6,9.5){\circle*{.8}}\put(12.5,7.9){$\cc^1$}
\thinlines
\Line{0}{0}{12}{0}
\Line{0}{0}{0}{14.7}
\Line{12}{0}{17}{18}
\Line{0}{12}{17}{18}
\Line{8}{0}{13.5}{19.8}
\Line{0}{14.7}{13.5}{19.8}
\Line{17}{18}{13.5}{19.8}
\put(13.1,16.3){\vector(2,3){1}}
\Line{0}{0}{7.7}{16.1}
\multiput(0,0)(0.286, 0.362){50}{\makebox(0,0){.}}
\put(6.5,13.7){\vector(1,2){1}}
\Line{0}{0}{12.6}{9.5}\put(11.2,8.4){\vector(4,3){1}}
\put(4,-3.5){overlap}
\end{picture}}}
\caption{View of $\tilde E$ from the top.
The two cases are for $\cPj^3(f_1)$ and $\cPj^3(f_3)$ overlap or not.}\label{fig:topview2}
\end{figure}

\subsection{Invertibility of $\C\circ\hh$}

We have affinely mapped our convex, cuboidal hexahedron $E$ to $\tilde E$.  An affine transformation
will take parallel lines to parallel lines. Therefore, if $E$ has two pair of parallel faces, or if
E is a truncated pillar, the same will be true of $\tilde E$.

\begin{theorem}\label{thm:a6}
For a convex, cuboidal hexahedron $\tilde E$, if one pair of opposite faces are parallel, then
the matrix $\C\circ\hh$ is invertible.
\end{theorem}
\begin{proof}
Without loss of generality, we assume that face 4 is parallel to face 5.
Therefore $\nu_{5,1} = \nu_{5,2} = 0$ and by Theorem~\ref{thm:minors} we know that
$\nu_{5,3}>0$. 
The invertibility of matrix $\C\circ\hh$ is reduced to showing that
\begin{equation}\label{eq:cnu12}
\det\left[\begin{matrix}
c_{1}^1 \nu_{1,1} & c_{1}^3 \nu_{3,1} \\
c_{2}^1 \nu_{1,2} & c_{2}^3 \nu_{3,2}  
\end{matrix}\right] > 0.
\end{equation}
By Lemma~\ref{lem:a4}, we know that $c_1^1c_2^3>c_2^1c_1^3>0$. By Theorem~\ref{thm:minors},
we have $\nu_{1,1}\nu_{3,2} > \nu_{1,2}\nu_{3,1}$.
Therefore, $c_1^1c_2^3\nu_{1,1}\nu_{3,2} > c_2^1c_1^3\nu_{1,2}\nu_{3,1}$, and (\ref{eq:cnu12}) holds.
\end{proof}

\begin{theorem}\label{thm:pillar}
For any truncated pillar $\tilde E$, the matrix $\C\circ\hh$ is invertible.
\end{theorem}

\begin{proof}
We assume without loss of generality that $\tilde E$ is a truncated vertical pillar, so
$\nu_{0,3}=\nu_{1,3}=\nu_{2,3}=\nu_{3,3}=0$. The matrix $\C\circ\hh$ reduces to
\begin{equation}
\C\circ\hh = \left[\begin{matrix}
c_{1}^1 \nu_{1,1} & c_{1}^3 \nu_{3,1} & c_{1}^5 \nu_{5,1}\\
c_{2}^1 \nu_{1,2} & c_{2}^3 \nu_{3,2} & c_{2}^5 \nu_{5,2}\\
0 & 0 & c_{3}^5 \nu_{5,3}
\end{matrix}\right].
\end{equation}
Moreover, the projection of $\cc^1$ on the bottom plane is in the line from
$\x_{124}$ to $\x_{134}$, and the projection of $\cc^3$ in the line from
$\x_{134}$ to $\x_{034}$
(see Figure~\ref{fig:topview2}, where now $\x_{035}$ and $\x_{034}$
are on top of each other, as are $\x_{134}$ and $\x_{135}$, and also $\x_{124}$ and $\x_{125}$).
Therefore, we have
\begin{equation}
\det\left[\begin{matrix}
c_{1}^1 \ &\ c_{1}^3 \\
c_{2}^1 \ &\ c_{2}^3  
\end{matrix}\right] > 0.
\end{equation}
The rest of the proof follows that of Theorem~\ref{thm:a6}.
\end{proof}

\section{Proof of Theorem~\ref{thm:simpleAT1fluxes-symm}}\label{sec:appb}

For $\tilde E$, the local variables on face 1 are $x_2$ and $x_3$, so
a base for the normal flux on $f_1$ is ${\rm span}\{1, x_2, x_3\} =
\Po_1(f_1)$. Similarly, ${\rm span}\{1, x_1, x_3\} = \Po_1(f_3)$, 
and ${\rm span}\{1,x_1,x_2\}=\Po_1(f_5)$.
Define the operator $\myStrut\cF_{135}^* \in \Re^{1\times9}$ to be the normal fluxes
of $f_1$, $f_3$, and $f_5$ in the local degrees of freedom, i.e., 
\begin{align}
\cF_{135}^*(\uu){\bf X}^T = \cF_{135}(\uu),\quad\text{where}\quad
{\bf X} = \left[\begin{array}{ccc|ccc|ccc}
1 & x_2 & x_3 & 1 & x_1 & x_3 & 1 & x_1 & x_2
\end{array}\right].
\end{align}
Similarly, we define 
$
\cF_{135}^*(\uu_1,\ldots,\uu_n) = \left({\small
\begin{matrix}
\cF_{135}^*(\uu_1)\\
\vdots\\
\cF_{135}^*(\uu_n)
\end{matrix}
}\right)\in \Re^{n\times9}.
$ 
On $f_1$, $(\x-\e_1)\cdot \nu_1=0$, i.e., 
$x_1\nu_{1,1} = \nu_{1,1} - x_2\nu_{1,2}-x_3\nu_{1,3}$.
Similar statements hold on $f_3$ and $f_5$, so
we can rewrite (\ref{eq:flux9-11}) as 
\begin{align}\label{eq:fstar9-11}
\cF_{135}^*(\bpsi_{9}^*,\bpsi_{10}^*,\bpsi_{11}^*)
= \left[\begin{array}{ccc|ccc|ccc}
\nu_{1,1} & -\nu_{1,2} & -\nu_{1,3} & 0 & \nu_{3,1} & 0 & 0 & \nu_{5,1} & 0\\
0 & \nu_{1,2} & 0 & \nu_{3,2} & -\nu_{3,1} & -\nu_{3,3} & 0 & 0 & \nu_{5,2}\\
0 & 0 & \nu_{1,3} & 0 & 0 & \nu_{3,3} & \nu_{5,3} & -\nu_{5,1} & -\nu_{5,2}
\end{array}\right].
\end{align}

To prove that (\ref{eq:simpleAT1fluxes-symm}) provides independent degrees of freedom, 
we need to show that the $9\times 9$ matrix
\begin{align}\nonumber
&\cF_{135}^*(\bpsi_{9}^*,\ldots,\bpsi_{11}^*, \bsigma_0, \ldots,\bsigma_5)
\\
& \quad = \left[\begin{array}{ccc | ccc | ccc}\label{eq:fstar-all}
\nu_{1,1} & -\nu_{1,2} & -\nu_{1,3} \ & 0 & \nu_{3,1} & 0 & 0 & \nu_{5,1} & 0\\
0 & \nu_{1,2} & 0 & \ \nu_{3,2} & -\nu_{3,1} & -\nu_{3,3} \ & 0 & 0 & \nu_{5,2}\\
0 & 0 & \nu_{1,3} & 0 & 0 & \nu_{3,3} & \ \nu_{5,3} & -\nu_{5,1} & -\nu_{5,2}\\
\hline
-c_2^1 & 1 & 0 & 0 & 0 & 0 & 0 & 0 & 0\myStrut\\
-c_3^1 & 0 & 1 & 0 & 0 & 0 & 0 & 0 & 0\myStrut\\
\hline
0 & 0 & 0 & -c_1^3 & 1 & 0 & 0 & 0 & 0\myStrut\\
0 & 0 & 0 & -c_3^3 & 0 & 1 & 0 & 0 & 0\myStrut\\
\hline
0 & 0 & 0 & 0 & 0 & 0 & -c_1^5 & 1 & 0\myStrut\\
0 & 0 & 0 & 0 & 0 & 0 & -c_2^5 & 0 & 1\myStrut
\end{array}\right]
\end{align}
is invertible. By the fact that $\cc^i = [c_1^i \:\: c_2^i \:\: c_3^i]^T$ is on $f_i$,
$i=1,3,5$,  
we know that, e.g., $c_1^1\nu_{1,1} = \nu_{1,1} - c_2^1\nu_{1,2} - c_3^1\nu_{1,3}$.
In (\ref{eq:fstar-all}), using rows $4$ to $9$ to cancel out entries in columns 2, 3, 
5, 6, 8, and 9 of
the first three rows, we obtain
\begin{align}\nonumber
\left[\begin{array}{ccc|ccc|ccc}
c_1^1\nu_{1,1} \ &\ 0 \ &\ 0 \ &\ c_1^3\nu_{3,1} \ &\  0 \ &\ 0 \ &\ c_1^5\nu_{5,1} \ &\ 0 \ &\ 0\\
c_2^1\nu_{1,2} & 0 & 0 & c_2^3\nu_{3,2} & 0 & 0 & c_2^5\nu_{5,2} & 0 & 0\\
c_3^1\nu_{1,3} & 0 & 0 & c_3^3\nu_{3,3} & 0 & 0 & c_3^5\nu_{5,3} & 0 & 0\\
\hline
-c_2^1 & 1 & 0 & 0 & 0 & 0 & 0 & 0 & 0\myStrut\\
-c_3^1 & 0 & 1 & 0 & 0 & 0 & 0 & 0 & 0\myStrut\\
\hline
0 & 0 & 0 & -c_1^3 & 1 & 0 & 0 & 0 & 0\myStrut\\
0 & 0 & 0 & -c_3^3 & 0 & 1 & 0 & 0 & 0\myStrut\\
\hline
0 & 0 & 0 & 0 & 0 & 0 & -c_1^5 & 1 & 0\myStrut\\
0 & 0 & 0 & 0 & 0 & 0 & -c_2^5 & 0 & 1\myStrut
\end{array}\right].
\end{align}
We rearrange the columns to
\begin{align}\nonumber
\left[\begin{array}{ccc|cccccc}
c_1^1\nu_{1,1} \ &\ c_1^3\nu_{3,1} \ &\ c_1^5\nu_{5,1} \ &\ 0 \ &\ 0 \ &\ 0 \ &\ 0 \ &\ 0 \ &\ 0\\
c_2^1\nu_{1,2} & c_2^3\nu_{3,2} & c_2^5\nu_{5,2} & 0 & 0 &  0 & 0  & 0 & 0\\
c_3^1\nu_{1,3} & c_3^3\nu_{3,3} & c_3^5\nu_{5,3} & 0 & 0 &  0 & 0  & 0 & 0\\
\hline
-c_2^1 & 0 & 0 & 1 & 0  & 0 & 0  & 0 & 0\myStrut\\
-c_3^1 & 0 & 0 & 0 & 1  & 0 & 0  & 0 & 0\myStrut\\
0 & -c_1^3 & 0 & 0 & 0 & 1 & 0  & 0 & 0\myStrut\\
0 & -c_3^3 & 0 & 0 & 0 & 0 & 1  & 0 & 0\myStrut\\
0 & 0 & -c_1^5 & 0 & 0 & 0 & 0  & 1 & 0\myStrut\\
0 & 0 & -c_2^5 & 0 & 0 & 0 & 0  & 0 & 1\myStrut
\end{array}\right].
\end{align}
The upper left submatrix is exactly $\C\circ\hh$, and the proof of 
Theorem~\ref{thm:simpleAT1fluxes-symm} is complete.

\section{Proof of Theorem~\ref{thm:simpleAT1fluxes-nonsymm}}\label{sec:appc}

Rewrite (\ref{eq:simpleAT1fluxes-nonsymm}) with $\cF_{135}^*$, to obtain
\begin{align}\nonumber
&\cF_{135}^*(\bpsi_{9}^*,\ldots,\bpsi_{11}^*, \bsigma_0, \ldots,\bsigma_3, \bsigma_4^*, \bsigma_5^*)
\\
\label{eq:fstar-all2}
& = \left[\begin{array}{ccc | ccc | ccc}
\nu_{1,1} & -\nu_{1,2} & -\nu_{1,3} & 0 & \nu_{3,1} & 0 & 0 & \nu_{5,1} & 0\\
0 & \nu_{1,2} & 0 & \nu_{3,2} & -\nu_{3,1} & -\nu_{3,3}  & 0 & 0 & \nu_{5,2}\\
0 & 0 & \nu_{1,3} & 0 & 0 & \nu_{3,3} &  \nu_{5,3} & -\nu_{5,1} & -\nu_{5,2}\\
\hline
-c_2^1 & 1 & 0 & 0 & 0 & 0 & 0 & 0 & 0\myStrut\\
-c_3^1 & 0 & 1 & 0 & 0 & 0 & 0 & 0 & 0\myStrut\\
\hline
0 & 0 & 0 & -c_1^3 & 1 & 0 & 0 & 0 & 0\myStrut\\
0 & 0 & 0 & -c_3^3 & 0 & 1 & 0 & 0 & 0\myStrut\\
\hline
\dfrac{-|f_5|c_1^5+t\myStrut}{|f_1|} & 0 & 0 & \dfrac{-t}{|f_3|} & 0 & 0 & 0 & 1 & 0\\[7pt]
\dfrac{-s}{|f_1|} & 0 & 0 & \dfrac{-|f_5|c_2^5+s\myStrut}{|f_3|} & 0 & 0 & 0 & 0 & 1
\end{array}\right].
\end{align}
If there exist constants $s$ and $t$ such that the matrix (\ref{eq:fstar-all2})
is invertible, the non-symmetric supplements $\bsigma_0$ to $\bsigma_3$, $\bsigma_4^*$,
and $\bsigma_5^*$ provide independent degrees of 
freedom.

Using rows $4$ to $7$ to cancel out entries in columns 2, 3, 5, and 6 in the first three rows, we obtain
\begin{align}\nonumber
\left[\begin{array}{ccc | ccc | ccc}
c_1^1\nu_{1,1} & 0 & 0 & c_1^3\nu_{3,1} & 0 & 0 & 0 & \nu_{5,1} & 0\\
c_2^1\nu_{1,2} & 0 & 0 & c_2^3\nu_{3,2} & 0 & 0 & 0 & 0 & \nu_{5,2}\\
c_3^1\nu_{1,3} & 0 & 0 & c_3^3\nu_{3,3} & 0 & 0 & \nu_{5,3} & -\nu_{5,1} & -\nu_{5,2}\\
\hline
-c_2^1 &\ 1\ &\ 0\ & 0 &\ 0\ &\ 0\ & 0 &\ 0\ &\ 0\myStrut\\
-c_3^1 & 0 & 1 & 0 & 0 & 0 & 0 & 0 & 0\myStrut\\
\hline
0 & 0 & 0 & -c_1^3 & 1 & 0 & 0 & 0 & 0\myStrut\\
0 & 0 & 0 & -c_3^3 & 0 & 1 & 0 & 0 & 0\myStrut\\
\hline
(-|f_5|c_1^5+t\myStrut)/{|f_1|} & 0 & 0 & {-t}/{|f_3|} & 0 & 0 & 0 & 1 & 0\\
{-s}/{|f_1|} & 0 & 0 & (-|f_5|c_2^5+s\myStrut)/{|f_3|} & 0 & 0 & 0 & 0 & 1
\end{array}\right].
\end{align}
Rearrange the columns and rows to see
\begin{align}\nonumber
\left[\begin{array}{cccc|ccccc}
c_1^1\nu_{1,1} & c_1^3\nu_{3,1} & \nu_{5,1} & 0 & 0 & 0 & 0 & 0 & 0 \\
c_2^1\nu_{1,2} & c_2^3\nu_{3,2} & 0 & \nu_{5,2} & 0 & 0 & 0 & 0 & 0 \\[2pt]
(-|f_5|c_1^5+t\myStrut)/{|f_1|} & {-t}/{|f_3|} & 1 & 0 & 0 & 0 & 0 & 0 & 0\\[2pt]
{-s}/{|f_1|} & (-|f_5|c_2^5+s\myStrut)/{|f_3|} & 0 & 1 & 0 & 0 & 0 & 0 & 0\\[2pt]
\hline
c_3^1\nu_{1,3} & c_3^3\nu_{3,3} & -\nu_{5,1} & -\nu_{5,2} & \nu_{5,3}& 0 & 0 & 0 & 0\myStrut\\
-c_2^1 &\ 0\ &\ 0\ & 0 &\ 0\ &\ 1\ & 0 &\ 0\ &\ 0\myStrut\\
-c_3^1 & 0 & 0 & 0 & 0 & 0 & 1 & 0 & 0\myStrut\\
0 & -c_1^3 & 0 & 0 & 0 & 0 & 0 & 1 & 0\myStrut\\
0 & -c_3^3 & 0 & 0 & 0 & 0 & 0 & 0 & 1\myStrut
\end{array}\right].
\end{align}
This matrix is invertible if and only if
\begin{align}\label{eq:red4}
\left[\begin{array}{cc|cc}
c_1^1\nu_{1,1} & c_1^3\nu_{3,1} & \nu_{5,1} & 0 \\[2pt]
c_2^1\nu_{1,2} & c_2^3\nu_{3,2} & 0 & \nu_{5,2} \\[2pt]
\hline
(-|f_5|c_1^5+t\myStrut)/{|f_1|} & {-t}/{|f_3|} & 1 & 0 \\[2pt]
{-s}/{|f_1|} & (-|f_5|c_2^5+s\myStrut)/{|f_3|} & 0 & 1 \\[2pt]
\end{array}\right]
\end{align}
is invertible.
A $2\times 2$ block matrix has the following lemma ~\cite{Silvester_2000}.
\begin{lemma}
If ${\bf M} = \left(\begin{matrix}
\bf A \ & \bf B\\
\bf C & \bf D
\end{matrix}\right)$, where $\bf A$, $\bf B$, $\bf C$, ${\bf D}\in\Re^{n\times n}$
and $\bf CD = DC$, then 
$$\det {\bf M} = \det ({\bf AD-BC}).$$
\end{lemma}
Obviously, the lower right submatrix of (\ref{eq:red4}) (an identity matrix)
commutes with any $2\times 2$ matrix. Thus, to prove that matrix (\ref{eq:red4}) is invertible,
we need to show that
\begin{align}\label{eq:red2}
\det\left[\begin{matrix}
c_1^1\nu_{1,1} + \nu_{5,1}(|f_5|c_1^5-t\myStrut)/{|f_1|} 
& c_1^3\nu_{3,1}+{t}\,\nu_{5,1}/{|f_3|}\\[2pt]
c_2^1\nu_{1,2}  + s\,\nu_{5,2}/{|f_1|}
& c_2^3\nu_{3,2} + \nu_{5,2}(|f_5|c_2^5-s\myStrut)/{|f_3|}
\end{matrix}\right]\neq 0.
\end{align}
This determinant is a bilinear function in $s$ and $t$, denoted as $d(s,t)$.
If we can prove that $d(s,t)\not\equiv 0 $, then we can find a pair
$(s^*, t^*)$ such that $d(s^*, t^*)\neq 0$, and the last two non-symmetric supplements
$\bsigma_4^*$ and $\bsigma_5^*$ in (\ref{eq:simpleAT1fluxes-nonsymm}) are defined.
There are two cases.

{\bf{Case 1}:} $\nu_{5,1} = 0$ and $\nu_{5,2} = 0$.
In this case, $\tilde E$ is a truncated vertical pillar, and by the proof of Theorem~\ref{thm:pillar},
we know that
\begin{align}
d(s,t) = \det\left[\begin{matrix}
c_1^1\nu_{1,1} 
& c_1^3\nu_{3,1}\\[2pt]
c_2^1\nu_{1,2}  
& c_2^3\nu_{3,2}
\end{matrix}\right]> 0,
\end{align}
and $s$ and $t$ may be taken arbitrarily.

{\bf{Case 2}:} $\nu_{5,2}\neq 0$ or $\nu_{5,1}\neq 0$.
By symmetry, we only show the situation $\nu_{5,2}\neq 0$ here.
Let
\begin{align}
a &= d(0,0) = \det\left[\begin{matrix}
c_1^1\nu_{1,1} + \nu_{5,1}(|f_5|c_1^5)/{|f_1|}
& c_1^3\nu_{3,1}\\[2pt]
c_2^1\nu_{1,2}  
& c_2^3\nu_{3,2} + \nu_{5,2}(|f_5|c_2^5)/{|f_3|}
\end{matrix}\right],\\
b &= d(|f_5|c_2^5,0) = \det\left[\begin{matrix}
c_1^1\nu_{1,1} + \nu_{5,1}(|f_5|c_1^5)/{|f_1|}
& \quad c_1^3\nu_{3,1}\\[2pt]
c_2^1\nu_{1,2}  + \nu_{5,2}(|f_5|c_2^5)/{|f_1|}
& \quad c_2^3\nu_{3,2} 
\end{matrix}\right].
\end{align}
Then
\begin{align}\nonumber
a - b & = (\nu_{5,2}|f_5|c_2^5)\det\left[\begin{matrix}
c_1^1\nu_{1,1} + \nu_{5,1}(|f_5|c_1^5)/{|f_1|}
& \quad c_1^3\nu_{3,1}\\[2pt]
-1/|f_1|
& \quad 1/|f_3|
\end{matrix}\right]\\
&= \frac{\nu_{5,2}|f_5|c_2^5}{|f_1||f_3|}
(|f_1|c_1^1\nu_{1,1} + |f_3|c_1^3\nu_{3,1}+|f_5|c_1^5\nu_{5,1}) \neq 0,
\end{align}
since
\begin{align}
&|f_1|c_1^1\nu_{1,1} + |f_3|c_1^3\nu_{3,1}+|f_5|c_1^5\nu_{5,1} \\\nonumber
&\quad = \int_{\partial \tilde E} \threeVec{x_1}{0}{0}\cdot\nu dA 
 = \int_{\tilde E}\div\threeVec{x_1}{0}{0} dV = |\tilde E| \neq 0.
\end{align}
The fact $a\neq b$ implies that $d(s,t)\not\equiv 0$, and so \eqref{eq:fstar-all2} is invertible.
\endAdd




\end{document}